\documentclass[12pt,a4paper]{amsart}

\usepackage[utf8]{inputenc}
\usepackage[T1]{fontenc}
\usepackage[normalem]{ulem}
\usepackage{cancel}
\usepackage{verbatim} % F�r Kommentare ueber mehrere Zeilen
\usepackage{bm} % bold symbols

\usepackage[pdftex]{graphicx}
\usepackage{latexsym}
\usepackage{amsmath,amssymb,amsthm}
\usepackage[shortlabels]{enumitem}   % f�r andere Z�hlarten bei enumerate

\usepackage{mathpazo} % f�r Schrift
\usepackage{nicefrac} % f�r schr�ge Br�che
\allowdisplaybreaks % f�r Seitenumbruch in align-Umgebung
\usepackage{url}

\usepackage{scalerel,stackengine}
\stackMath
\newcommand\reallywidehat[1]{%
\savestack{\tmpbox}{\stretchto{%
  \scaleto{%
    \scalerel*[\widthof{\ensuremath{#1}}]{\kern-.6pt\bigwedge\kern-.6pt}%
    {\rule[-\textheight/2]{1ex}{\textheight}}%WIDTH-LIMITED BIG WEDGE
  }{\textheight}% 
}{0.5ex}}%
\stackon[1pt]{#1}{\tmpbox}%  f�r reallywidehat
}

\usepackage[Option]{overpic}
\usepackage{tikz}
\usetikzlibrary{positioning,arrows}
\tikzset{
  arrow/.style={-latex, shorten >=1ex, shorten <=1ex}}

\newtheorem{Theorem}{Theorem}
\newtheorem{Lemma}{Lemma}		
\newtheorem{Proposition}{Proposition}
\newtheorem{Corollary}{Corollary}	
\theoremstyle{definition}
\newtheorem{Remark}{Remark}	 
\newtheorem{Definition}{Definition} 
   
\newcommand{\C}{\mathbb{C}} % komplexe
\newcommand{\R}{\mathbb{R}} % reelle
\newcommand{\Z}{\mathbb{Z}} % ganze
\newcommand{\N}{\mathbb{N}} % natuerliche

\newcommand{\RT}{\operatorname{Re}}
\newcommand{\IT}{\operatorname{Im}}
\newcommand{\iu}{\mathrm{i}}
\newcommand{\eu}{\mathrm{e}}
\newcommand{\per}{\mathrm{per}}
\newcommand{\dom}{\mathrm{dom}}

\renewcommand{\epsilon}{\varepsilon} % schoenes epsilon
\DeclareMathOperator{\sign}{sign}
\DeclareMathOperator{\spann}{span}
\DeclareMathOperator{\range}{range}

\makeatletter\@addtoreset {equation}{section}\makeatother

\usepackage[colorlinks=true,urlcolor=blue]{hyperref} 			% Referenzen klickbar, als letztes laden

\begin{document}

\title[Pinning in the extended Lugiato-Lefever equation]{Pinning in the extended Lugiato-Lefever equation}

\author{Lukas Bengel}
\address{L. Bengel \hfill\break 
Institute for Analysis,\hfill\break
Karlsruhe Institute of Technology (KIT), \hfill\break
D-76128 Karlsruhe, Germany}
\email{lukas.bengel@kit.edu}

\author{Dmitry Pelinovsky}
\address{D. Pelinovsky \hfill\break
Department of Mathematics and Statistics, \hfill\break
McMaster University, \hfill\break
Hamilton, Ontario, Canada, L8S 4K1}
\email{dmpeli@math.mcmaster.ca}

\author{Wolfgang Reichel}
\address{W. Reichel \hfill\break 
Institute for Analysis,\hfill\break
Karlsruhe Institute of Technology (KIT), \hfill\break
D-76128 Karlsruhe, Germany}
\email{wolfgang.reichel@kit.edu}

\begin{abstract} 
	We consider a variant of the Lugiato-Lefever equation (LLE), which is a nonlinear Schr\"odinger equation on a one-dimensional torus with forcing and damping, to which we add a first-order derivative term with a potential $\epsilon V(x)$. The potential breaks the translation invariance of LLE. Depending on the existence of zeroes of the effective potential $V_\text{eff}$, which is a suitably weighted and integrated version of $V$, we show that stationary solutions from $\epsilon=0$ can be continued locally into the range $\epsilon\not =0$. Moreover, the extremal points of the $\epsilon$-continued solutions are located near zeros of $V_\text{eff}$. We therefore call this phenomenon \emph{pinning} of stationary solutions. If we assume additionally that the starting stationary solution at $\epsilon=0$ is spectrally stable with the simple zero eigenvalue due to translation invariance being the only eigenvalue on the imaginary axis, we can prove asymptotic stability or instability of its $\epsilon$-continuation depending on the sign of $V_\text{eff}'$ at the zero of $V_\text{eff}$ and the sign of $\epsilon$. The variant of the LLE arises in the description of optical frequency combs in a Kerr nonlinear ring-shaped microresonator which is pumped by two different continuous monochromatic light sources of different frequencies and different powers. Our analytical findings are illustrated by numerical simulations.   
\end{abstract}

\keywords{Nonlinear Schr\"odinger equation, bifurcation theory, continuation method}

\date{\today} 
	
\subjclass[2000]{Primary: 34C23, 34B15; Secondary: 35Q55, 34B60}

	%34C23 (1991-now) Bifurcation theory for ordinary differential equations
    %34B15 (1973-now) Nonlinear boundary value problems for ordinary differential equations 	
    %35Q55 (1991-now) NLS equations (nonlinear Schr�dinger equations)
    %34B60 (2000-now) Applications of boundary value problems involving ordinary differential equations 
    
\maketitle

\section{Introduction} The Lugiato-Lefever equation \cite{Lugiato_Lefever1987} is the most commonly used model to describe electromagnetic fields inside a resonant cavity that is pumped by a strong continuous laser source. Inside the cavity the electromagnetic field propagates and suffers losses due to curvature and/or material imperfections. Most importantly, the cavity consists of a Kerr-nonlinear material so that triggered by modulation instability the field may experience a nonlinear interaction of the pumped and resonantly enhanced modes of the cavity. Under appropriate driving conditions of the resonant cavity and the laser, a stable Kerr-frequency comb may form in the cavity, which is a spatially localized and spectrally broad waveform. 

Since their discovery by the 2005 noble prize laureate Theodor H\"ansch, frequency combs have seen an enormously wide field of applications, e.g., in high capacity optical communications \cite{marin2017microresonator}, ultrafast optical ranging \cite{trocha2018ultrafast}, optical frequency metrology \cite{Udem2002}, or spectroscopy \cite{picque2019frequency,yang2017microresonator}. The Lugiato-Lefever equation (LLE) is an amplitude equation for the electromagnetic field inside the cavity derived by means of the slowly varying envelope approximation. 

In the following we assume that the cavity is a ring-shaped microresonator with normalized perimeter $2\pi$. Using dimensionless quantities and writing $u(x,t)=\sum_{k\in\Z} u_k(t)\eu^{\iu kx}$ for the slowly varying and $2\pi$-periodic amplitude of the electromagnetic field, the LLE in its original form \cite{Lugiato_Lefever1987} reads as  
\begin{equation}\label{LLE_original}
\mathrm{i} \partial_t u =- d \partial_x^2 u+(\zeta-\mathrm{i}\mu)u-|u|^2u+ \mathrm{i}f_0, \qquad (x,t) \in \mathbb{T} \times \mathbb{R},
\end{equation}
where $\mathbb{T}$ is a circle of length $2\pi$. The dispersion relation 
for the $k$-th Fourier mode of the resonator is given in the form
$\omega_k = \omega_0+d_1k+d_2k^2$ with $d := \frac{2}{\kappa}d_2$ being the normalized dispersion coefficient and $\kappa>0$ being the cavity decay rate. The detuning value $\zeta$ represents the off-set between the laser frequency $\omega_{p_0}$ and the closest resonance frequency $\omega_0$ of the zero-mode $k_0=0$ of the resonator, and the value $\mu$ quantifies the damping coefficient. Finally, $f_0$ stands for pump strength with power $|f_0|^2$. 

More recently, novel pumping schemes have been discussed \cite{Taheri_2017}, where instead of one monochromatic laser pump one uses a dual laser pump with two different frequencies as a source term. Using again dimensionless quantities the resulting equation is given by 
\begin{equation}\label{LLE_dual}
\mathrm{i} \partial_t u =- d \partial_x^2 u+(\zeta-\mathrm{i}\mu)u-|u|^2u+ \mathrm{i}f_0 + \mathrm{i}f_1\mathrm{e}^{\mathrm{i}(k_1 x-\nu_1 t)}, \qquad (x,t) \in \mathbb{T} \times \mathbb{R},
\end{equation}
cf. \cite{Gasmi_Jahnke_Kirn_Reichel,Gasmi_Peng_Koos_Reichel,Taheri_2017} for a detailed derivation. In contrast to \eqref{LLE_original} there is now a second source term with pump strength $f_1$ and $k_1$ stands for the second pumped mode (the first pumped mode is again $k_0=0$). This gives rise to two detuning variables $\zeta=\frac{2}{\kappa}(\omega_0-\omega_{p_0})$, $\zeta_1=\frac{2}{\kappa}(\omega_{k_1}-\omega_{p_1})$ and they define $\nu_1=\zeta-\zeta_1+d k_1^2$. One of the main outcomes of \cite{Gasmi_Peng_Koos_Reichel} is that the stationary states of \eqref{LLE_dual} are far more localized than the stationary states of \eqref{LLE_original}, and the best results can be achieved when $f_0=f_1$ among all power distributions such that $f_0^2+f_1^2$ is kept constant. 

However, there are cases where a power distribution $|f_0|\gg |f_1|$ is more adequate in physical experiments. In this case, it is shown in Appendix \ref{appA} that one can derive from \eqref{LLE_dual} the perturbed LLE 
in the form
\begin{equation}\label{TWE_dyn}
\iu \partial_t u = -d \partial_{x}^2 u + \iu \epsilon V(x) \partial_x u +(\zeta-\iu \mu) u -|u|^2 u + \iu f_0, \qquad (x,t) \in \mathbb{T} \times \mathbb{R},
\end{equation}
where in the physical context $V(x)=\omega_1-2dk_1^2\frac{f_1}{f_0}\cos(x)$ and $\epsilon=1$. However, if $\omega_1$ and $k_1^2f_1/f_0$ are small, we will consider \eqref{TWE_dyn} as the perturbed LLE with $\epsilon\in \R$ being small and $V\in C^1([-\pi,\pi],\R)$ being a generic periodic potential. Recall that \eqref{TWE_dyn} is already set in a moving coordinate frame. In its stationary form the equation becomes
\begin{equation}\label{TWE}
-d u''+\mathrm{i} \epsilon V(x) u'+(\zeta-\mathrm{i}\mu)u-|u|^2 u+\mathrm{i}f_0=0, \qquad x \in \mathbb{T}.
\end{equation}
The main questions addressed in this paper are the existence and stability of the stationary solution of \eqref{TWE_dyn}. Our main results, which are stated in detail in Section~\ref{sec:results}, can be summarized as follows:
\begin{itemize} 
\item In Theorem~\ref{Fortsetzung_nichttrivial} we prove existence of solutions of \eqref{TWE} for small $\epsilon$ provided the effective potential $V_{\text{eff}}$ changes sign, where $V_{\text{eff}}$ is a weighted integrated version of the coefficient function $V$.
\item In Theorems~\ref{thm:spectral_stability} and ~\ref{thm:nonlinear_stability} we prove stability/instability properties of the solution obtained from Theorem~\ref{Fortsetzung_nichttrivial} with the time evolution of \eqref{TWE_dyn}.
\item In Section~\ref{sec:numerical} we illustrate the findings of our theorems by numerical simulations. The numerical simulations show that the location of the intensity extremum of the $\epsilon$-continued solutions does not change significantly for small $\epsilon$. Therefore, we call this phenomenon \emph{pinning of solutions at zeroes of the effective potential $V_\text{eff}$}.
\end{itemize}
Existence and bifurcation behavior of solutions  of \eqref{LLE_original} have been studied quite well, cf. 
\cite{Gaertner_et_al,gaertner_reichel_waves,Godey_2017,Godey_et_al2014,Mandel,Miyaji_Ohnishi_Tsutsumi2010,Parra-Rivas2018,Parra-Rivas2014,Parra-Rivas2016} and their stability properties have been investigated in \cite{hara_delcey,DelHara_periodic,Hakkaev_Stefanov,hara_johns_perk, hara_johns_perk_derijk,Perinet,Stanislavova_Stefanov}. Analytical and numerical investigations of \eqref{LLE_dual} have recently been reported \cite{Gasmi_Jahnke_Kirn_Reichel,Gasmi_Peng_Koos_Reichel}. In contrast, we are not aware of any treatment of \eqref{TWE_dyn}. However, a related problem, where instead of $\mathrm{i} \epsilon V(x) u'$ a term of the form $\epsilon V(x) u$ appears in the NLS equation, has been quite well studied, 
cf. \cite{Alama,Sigal,Kev}. In this case solutions are pinned near nondegenerate critical points of $V_\text{eff}$ instead of the zeroes of $V_\text{eff}$ as in our case. %on $\R^N$ in the semiclassical limit $d\to 0+$, cf. \cite{cingolani_et_al_1997,ambro_malch}

\section{Main results} \label{sec:results}

In this section we present our main results regarding existence and stability of stationary solutions of \eqref{TWE_dyn}. For $\epsilon=0$ there is a plethora of non-trivial (non-constant) stationary solutions, cf. \cite{Gaertner_et_al, Mandel}. We start with such a solution under the assumption of its non-degeneracy according to the following definition. 

\begin{Definition} A non-constant solution $u\in H_\per^2([-\pi,\pi],\C)$ of \eqref{TWE} for $\epsilon=0$ is called non-degenerate if the kernel of the linearized operator 
$$
L_u\varphi := -d \varphi'' +(\zeta-\iu\mu-2|u|^2)\varphi-u^2\bar\varphi,  \quad \varphi\in H_\per^2([-\pi,\pi],\C)
$$
consists only of $\spann\{u'\}$.
\label{non_degenerate_solution}
\end{Definition}

\begin{Remark} \label{fredholm_etc} Note that $L_u: H_\per^2([-\pi,\pi],\C) \to L^2([-\pi,\pi],\C)$ is a compact perturbation of the isomorphism $-d\partial_x^2+\sign(d): H_\per^2([-\pi,\pi],\C) \to L^2([-\pi,\pi],\C)$ and hence a Fredholm operator. Notice also that $\spann\{u'\}$ always belongs to the kernel of $L_u$ due to translation invariance in $x$ for $\epsilon = 0$. Non-degeneracy means that except for the obvious candidate $u'$ (and its real multiples) there is no other element in the kernel of $L_u$.  
\end{Remark}

One can ask the question whether non-constant non-degenerate solutions at $\epsilon=0$ in Definition \ref{non_degenerate_solution} may be continued into the regime of $\epsilon\not =0$. In order to describe the continuation, we denote such a solution by $u_0$ 
and its spatial translations by $u_\sigma(x):= u_0(x-\sigma)$. The non-degeneracy assumption implies that $\ker L_{u_{\sigma}}=\spann\{u_\sigma'\}$. Since the adjoint operator $L_{u_\sigma}^*$ also has a one-dimensional kernel there exists $\phi_\sigma^*\in H_\per^2([-\pi,\pi],\C)$ such $\ker L_{u_\sigma}^*=\spann\{\phi_\sigma^*\}$. Notice that $\phi_\sigma^*(x) = \phi_0^*(x-\sigma)$. 

Before stating our existence result, let us clarify the assumption on the potential $V$.
\begin{itemize}
\item[(A1)] The potential $V:[-\pi,\pi]\to \R,x \mapsto V(x)$ is a $2\pi$-periodic, continuously differentiable function.
\end{itemize}
The existence result is given by the following theorem.

\begin{Theorem} \label{Fortsetzung_nichttrivial} Let $d \in \R\setminus\{0\},f_0,\zeta,\mu \in \R$ be fixed and assume that (A1) holds. Let furthermore $u_0\in H_\per^2([-\pi,\pi],\C)$ be a non-constant, non-degenerate solution of \eqref{TWE} for $\epsilon=0$. If $\sigma_0$ is a simple zero of the function
\begin{equation} \label{eq:sigma_0}
\sigma \mapsto V_\mathrm{eff}(\sigma):= \RT \int_{-\pi}^{\pi} \iu V(x+\sigma) u_0'\bar\phi_0^*\,dx
\end{equation}
then there exists a continuous curve $(-\epsilon^\ast,\epsilon^\ast) \ni \epsilon \to u(\epsilon) \in H_\per^2([-\pi,\pi],\C)$ consisting of solutions of \eqref{TWE} with $\|u(\epsilon)-u_0(\cdot-\sigma_0)\|_{H^2} \leq C\epsilon$ for some constant $C>0$.
\end{Theorem}

\begin{Remark}
The value of $\sigma_0$ is determined from the existence of a unique solution $v\in H_\per^2([-\pi,\pi],\C)$ of the linear inhomogeneous equation
	$$
	L_{u_{\sigma_0}}v =-\iu V(x) u_{\sigma_0}'
	$$
with the property that $v\perp_{L^2} u_{\sigma_0}'$. Fredholm's condition shows that $\sigma_0$ is a zero of $V_{\mathrm{eff}}$. Simplicity of the zero of $V_{\mathrm{eff}}$ yields the result of Theorem \ref{Fortsetzung_nichttrivial}.
	\end{Remark}

To investigate the stability of a stationary solution $u$ we introduce the expansion
$$
u(x) + v(x,t) = u_1(x) + \iu u_2(x) + v_1(x,t) + \iu v_2(x,t)
$$
and substitute this into the perturbed LLE \eqref{TWE_dyn}. After neglecting the quadratic and cubic terms in $v$ and separating real and imaginary parts we obtain the linearized system for $\bm v = (v_1,v_2)$ which reads as
$$
\partial_t \bm v = \widetilde{L}_{u,\epsilon} \bm{v}
$$
and the linearization has the form
\begin{equation}
\label{decomposition}
\widetilde{L}_{u,\epsilon} = J A_u - I (\mu -\epsilon V(x) \partial_x)
\end{equation}
with
\begin{align*}
J :=
\begin{pmatrix}
0 & 1 \\
-1 & 0
\end{pmatrix},\;\;
I :=
\begin{pmatrix}
1 & 0 \\
0 & 1
\end{pmatrix}, \;\;
A_u:= 
\begin{pmatrix}
-d\partial_x^2 + \zeta - (3 u_1^2 + u_2^2) & -2u_1 u_2 \\
-2u_1 u_2 & -d\partial_x^2 + \zeta - (u_1^2 + 3u_2^2)
\end{pmatrix}.
\end{align*}
In the following we will often identify functions in $\C$ as vector-valued functions in $\R \times \R$ and use the notation 
$$
u = u_1 + \iu u_2 \in \C \quad\leftrightarrow\quad \bm u = \begin{pmatrix}u_1 \\ u_2 \end{pmatrix} \in \R^2.
$$
We denote the spectrum of $\widetilde{L}_{u,\epsilon}$ in $L^2([-\pi,\pi]) \times L^2([-\pi,\pi])$ by $\sigma(\widetilde{L}_{u,\epsilon})$ and the resolvent set 
	of $\widetilde{L}_{u,\epsilon}$ by $\rho(\widetilde{L}_{u,\epsilon})$.

For our stability results we require one additional spectral assumption on the non-degenerate solution $u_0$ regarding the spectrum of $\widetilde{L}_{u_0,0}$.
\begin{itemize}
\item[(A2)] The eigenvalue $0 \in \sigma(\widetilde{L}_{u_0,0})$ is algebraically simple and there exists $\xi > 0$ such that 
$$
\sigma(\widetilde{L}_{u_0,0}) \subset \{z\in \C: \RT z \leq -\xi\} \cup \{0\}. 
$$
\end{itemize}

\begin{Remark}
By Fredholm theory, the assumption of simplicity of the zero eigenvalue of $\widetilde{L}_{u_0,0}$ is equivalent to $\bm  u_0' \not\in \range \widetilde{L}_{u_0,0} = \spann\{J \bm\phi_0^*\}^\perp$. It will be convenient to use the normalization $\langle \bm u_0', J\bm\phi_0^* \rangle_{L^2} = \int_{-\pi}^\pi  \bm u_0' \cdot J\bm\phi_0^* \,dx = 1$. We also note that 
$$
\int_{-\pi}^\pi  \bm u_0' \cdot J\bm\phi_0^* \,dx = \RT \int_{-\pi}^{\pi} \iu u_0'\bar\phi_0^*\,dx.
$$
\label{rem-kernel}
 \end{Remark}

Before stating the stability results, let us clarify that $\ker L_{u}^*$ and $\ker L_{u}$ are linearly independent so that $V_{\mathrm{eff}}$ is generically nonzero. We also clarify the parity of eigenfunctions in $\ker L_{u}^*$ and $\ker L_{u}$ if $u_0$ is even in $x$. This is used for many practical computations.

\begin{Lemma}\label{lem:parity}
	Let $u_0 \in H^2_{\text{per}}([-\pi,\pi],\C)$ be a non-constant, non-degenerate solution of \eqref{TWE} for $\epsilon=0$. 
	Then the following holds:
	\begin{itemize} 
		\item[(i)] $u_0'$ and $\phi_0^*$ are linearly independent,
		\item[(ii)] if $u_0$ is even then $\phi_0^*$ is odd.
	\end{itemize}
\end{Lemma}

\begin{proof} Part (i): By using the decomposition (\ref{decomposition}) with $u = u_0$ and $\epsilon = 0$, the eigenvalue problems $L_{u_0} u_0' = 0$ and $L_{u_0}^* \phi_0^* = 0$ are equivalent to
	$$
	JA_{u_0}
	\begin{pmatrix}
	u_{01}' \\ u_{02}'
	\end{pmatrix} 
	= \mu
	\begin{pmatrix}
	u_{01}' \\ u_{02}'
	\end{pmatrix}
	,\quad\quad 
	JA_{u_0}
	\begin{pmatrix}
	\phi_{01}^* \\ \phi_{02}^*
	\end{pmatrix} 
	=-\mu
	\begin{pmatrix}
	\phi_{01}^* \\ \phi_{02}^*
	\end{pmatrix}.
	$$
	But since $(u_{01}',u_{02}')$ and $(\phi_{01}^*,\phi_{02}^*)$ are eigenvectors to the different eigenvalues $\mu$ and $-\mu$ of $JA_{u_0}$, respectively, they are linearly independent.
	
	\medskip
	
	Part (ii): By assumption we have that $\ker L_{u_0} = \spann\{u_0'\}$ and $u_0'$ is an odd function. Let us define the restriction of $L_{u_0}$ onto the odd functions%$H_{\text{per,odd}}^2:=\{\varphi\in H_{\text{per}}^2([-\pi,\pi],\C): \; \varphi(-x) = -\varphi(x)\}$  and the restriction
	$$
	L_{u_0}^\#:	H_{\text{per,odd}}^2  \to  L_{\text{per,odd}}^2,\,
	\varphi \mapsto L_{u_0} \varphi.
	$$
	Then $L_{u_0}^\#$ is again an index $0$ Fredholm operator with $\ker L_{u_0}^\#= \spann\{u_0'\}$. Further we have $(L_{u_0}^\#)^* = (L_{u_0}^*)^\#$ where
	$$
	(L_{u_0}^*)^\#:	H_{\text{per,odd}}^2 \to L_{\text{per,odd}}^2, \,
	\varphi  \mapsto L_{u_0}^* \varphi
	$$ 
	is the restriction of the adjoint onto the odd functions.
	But since $1=\dim \ker (L_{u_0}^*)^\# = \dim \ker L_{u_0}^*$ it follows that $\ker (L_{u_0}^*)^\# = \ker L_{u_0}^*$ and hence $\phi_0^* \in H_{\text{per,odd}}^2$ as claimed.
\end{proof}

% {\color{blue} Alternative (if we assume that $0$ is an algebraically simple eigenvalue):
% 
% Let us define the function $\tilde\phi(s):=-\phi^*(-s)$. Using that $u_0$ is an even function, it is easy to see that $L_{u_0}^*\tilde\phi = 0$ from which $\tilde\phi \in \ker(L_{u_0}^*)= \spann\{ \phi^*\}$ follows. But this means that $\phi^*$ must be either an odd or an even function and we have to exclude the latter. Since the zero eigenvalue is algebraically simple, we find that $\RT \int_0^{2\pi} u_0' \bar\phi^* \,dx \not = 0$ and from oddness of $u_0'$ we conclude that the non-vanishing of the integral is possible only if $\phi^*$ is odd as well.
% }

The stability results are given by the following two theorems. A stationary solution $u$ of \eqref{TWE} is called spectrally stable if $\RT(\lambda) \leq 0$ for all eigenvalues $\lambda$ of $\widetilde{L}_{u,\epsilon}$. It is called spectrally unstable if there exists one eigenvalue $\lambda$ with $\RT(\lambda)>0$. 

\begin{Theorem}\label{thm:spectral_stability} Let $d \in \R\setminus\{0\}, f_0,\zeta,\mu\in \R$ be fixed and assume that (A1) and (A2) hold. With $\sigma_0$ being a simple zero of $V_{\mathrm{eff}}$ as in Theorem~\ref{Fortsetzung_nichttrivial}, we have 
$$
V_{\mathrm{eff}}'(\sigma_0)=\RT \int_{-\pi}^{\pi} \iu V'(x+\sigma_0)  u_0'\bar\phi_0^* \,dx = \langle V'(\cdot + \sigma_0) \bm u_0', J\bm\phi_0^* \rangle_{L^2} \not = 0.
$$
Then there exists $\epsilon_0>0$ such that on the solution branch 
$(-\epsilon_0,\epsilon_0) \ni \epsilon \to u(\epsilon) \in  H_{\per}^2([-\pi,\pi],\C)$ of \eqref{TWE} with $u(0)=u_{\sigma_0}$ the solutions $u(\epsilon)$ are spectrally stable for $V'_{\mathrm{eff}}(\sigma_0) \cdot\epsilon >0$ and spectrally unstable for $V'_{\mathrm{eff}}(\sigma_0) \cdot\epsilon<0$. 
\end{Theorem}

\begin{Theorem}\label{thm:nonlinear_stability}
Let $u(\epsilon) \in H_\per^2([-\pi,\pi], \C)$ be a spectrally stable stationary solution of \eqref{TWE_dyn} for a small value of $\epsilon$ as in Theorem \ref{thm:spectral_stability}. Then $u(\epsilon)$ is asymptotically stable, i.e., there exist $\eta, \delta, C>0$ with the following properties. If $\varphi \in C([0,T),H^1_{\text{per}}([-\pi,\pi],\C))$ is a solution of \eqref{TWE_dyn} with maximal existence time $T$ and 
$$
\|\varphi(\cdot,0) - u(\epsilon) \|_{H^1} < \delta
$$ 
then $T=\infty$ and
$$
\|\varphi(\cdot,t) - u(\epsilon) \|_{H^1} \leq C \eu^{-\eta t} \|\varphi(\cdot,0) - u(\epsilon) \|_{H^1} \quad\text{for all } t\geq 0.
$$
\end{Theorem}

\begin{Remark}
	Due to periodicity of $V_{\mathrm{eff}}$ on $\mathbb{T}$, simple zeros of $V_{\mathrm{eff}}$ comes in pairs. By Theorems~\ref{thm:spectral_stability} and ~\ref{thm:nonlinear_stability} , one simple zero gives a solution branch consisting of asymptotically stable solutions for any sign of $\epsilon$. Moreover, at the bifurcation point $\epsilon = 0$ there is an exchange of stability, i.e., the zero eigenvalue crosses the imaginary axis with non zero speed. 
\label{rem:approx_stability}
\end{Remark}

\begin{Remark}
In \cite{DelHara_periodic,Hakkaev_Stefanov} the authors constructed spectrally stable solutions $u$ of \eqref{TWE} for $\epsilon=0$ in the case of anomalous dispersion $d>0$. These solutions satisfy the spectral condition $ \sigma(\widetilde{L}_{u,0}) \subset \{-2\mu\} \cup \{\RT z = -\mu\} \cup \{0\}$ and are therefore non-degenerate starting solutions for which our main results from Theorems~\ref{Fortsetzung_nichttrivial},  ~\ref{thm:spectral_stability}, and ~\ref{thm:nonlinear_stability} hold. 
\end{Remark}

\begin{Remark}
If $u$ is a solution of \eqref{TWE} then the relation 
$$
\int_{-\pi}^{\pi} (u'\bar{u}-\bar{u}'u) dx = 0
$$
holds. This constraint is satisfied by every even function $u$. In fact, the only solutions of equation \eqref{TWE} for $\epsilon = 0$ that we are aware of are even around $x=0$ (up to a shift).
\end{Remark}

\begin{Remark}
In the limit where $u_0$ is highly localized around $0$ (e.g.~the limit $d \to 0 \pm$) and the potential $V$ is wide, the effective potential $V_{\text{eff}}$ is well approximated by the actual potential $V$. More precisely we find the asymptotic
$$
V_{\text{eff}}(\sigma) = \RT \int_{-\pi}^\pi \iu V(x+\sigma) u_0' \bar{\phi}_0^* \, dx \approx V(\sigma) \RT \int_{-\pi}^\pi \iu  u_0' \bar{\phi}_0^* \, dx = V(\sigma)
$$
provided $\langle \iu u_0',\phi_0^*\rangle_{L^2}=1$. Thus, the asymptotically stable branch bifurcates from a simple zero $\sigma_0$ of $V$ with $V'(\sigma_0) \epsilon > 0$.
\label{rem:approx_potential}
\end{Remark}

\begin{Remark}
	 The criterion for stability of stationary solutions in Theorem \ref{thm:spectral_stability} can be written in a more precise form 
	 for small $\mu$ in the case of solitary waves. This limit 
	 is considered in Appendix \ref{appB}.
\end{Remark}

To summarize, our main results show that nondegenerate solutions of \eqref{TWE} for $\epsilon=0$ can be extended locally for small $\epsilon\not =0$ provided the effective potential $V_{\mathrm{eff}}$ has a sign-change. Depending on the derivative of $V_{\mathrm{eff}}$ at a simple zero we determined the stability properties of these solutions. It remains an open problem to give a criterion on $V$ or $V_{\mathrm{eff}}$ for the existence/stability of stationary solutions which applies when $|\epsilon|$ is large.

\section{Numerical simulations} \label{sec:numerical}

In the following we describe numerical simulations of solutions to \eqref{TWE}. We choose $f_0 = 2$, $\mu =1$, $V(x)=0.1+0.5\cos(x)$ and $d=\pm 0.1$. All computations are done with help of the Matlab package \texttt{pde2path} (cf. \cite{DOH14, UEC14}) which has been designed to numerically treat continuation and bifurcation in boundary value problems for systems of PDEs.

We begin with the description of the stationary solutions of the LLE \eqref{LLE_original}, which are the same as the solutions of \eqref{TWE} for $\epsilon=0$. The corresponding results are mainly taken from \cite{Gaertner_et_al, Mandel}. There is a curve of trivial, spatially constant solutions, cf. black line in Figure~\ref{fig:review}, and this is the same curve for anomalous dispersion ($d=0.1$) and normal dispersion ($d=-0.1$). Next one finds that there are finitely many bifurcation points on the curve of trivial solutions (blue dots). Depending on the sign of the dispersion parameter $d$ one can find now the branches of the single solitons on the periodic domain $\mathbb{T}$. In the following descriptions we always follow the path of trivial solutions by starting from negative values of $\zeta$. 

\begin{figure*}[h]
\centering
\begin{minipage}[t]{0.32\textwidth}
\begin{tikzpicture}[overlay]
\node at (2.655,2.07)
{\includegraphics[width=\columnwidth]{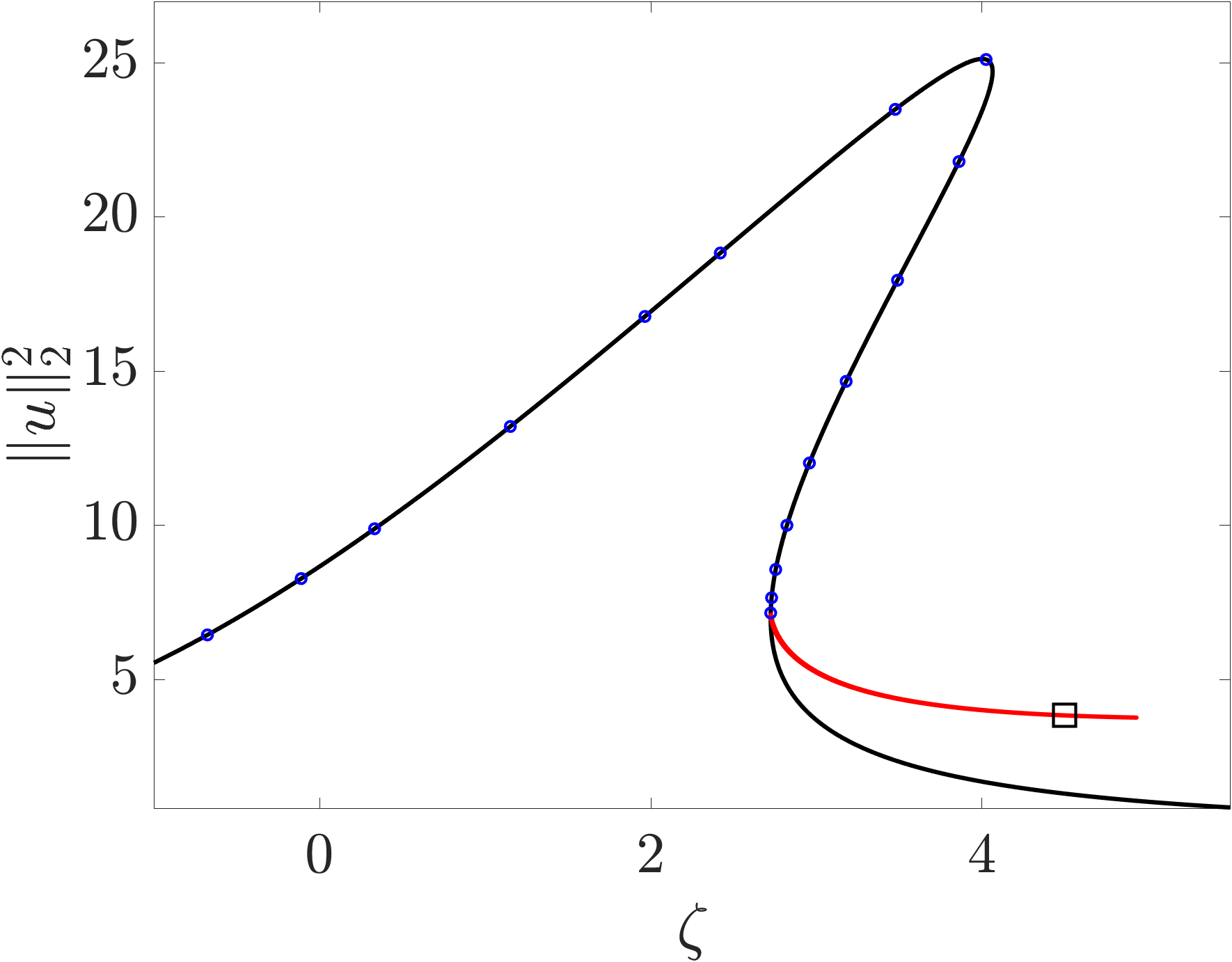}};
\node  at (4.55,2.55) {\includegraphics[width=0.3\columnwidth]{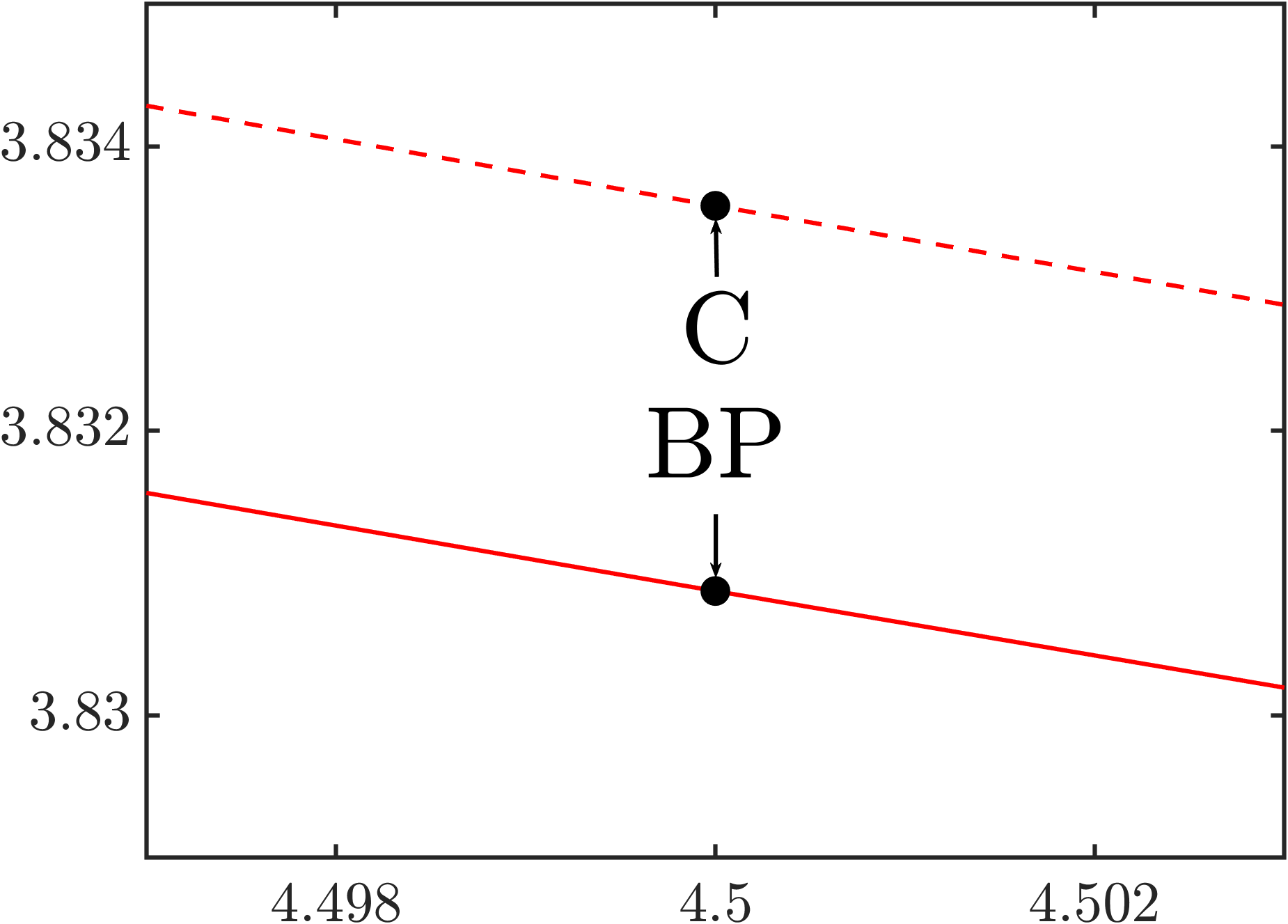}};
%\draw [draw=black] (2.95,1.26) rectangle (3.15,1.7);
\draw [draw=black] (4.54,1.11) -- (3.94,2.06);
\draw [draw=black] (4.634,1.11) -- (5.34,2.06);
\end{tikzpicture}
\end{minipage}\hspace*{0.5cm}
\begin{minipage}[t]{0.32\textwidth}
\includegraphics[width=\columnwidth]{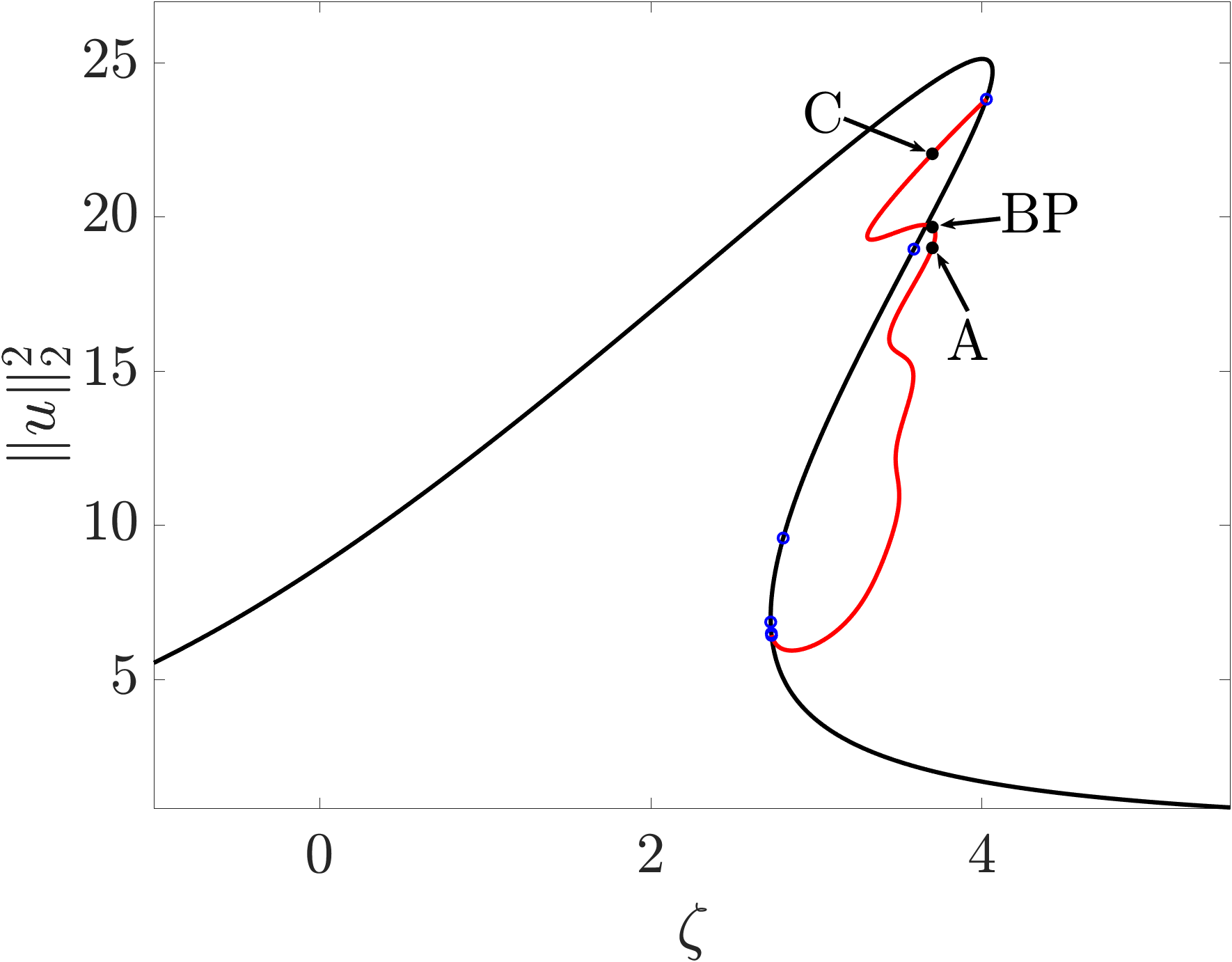}
\end{minipage}
\caption{Bifurcation diagram for the case $\epsilon = 0$. Blue dots indicate bifurcation points on the line of trivial solutions (black). The red curve denotes the single soliton solution branch. The point BP is chosen as a starting point for Theorem~\ref{Fortsetzung_nichttrivial}. Further solutions on the same branch for the same value of $\zeta$ are denoted by C (left panel) and A, C (right panel). Left panel for $d=0.1$, right panel for $d=-0.1$.}
\label{fig:review}
\end{figure*}

For $d=0.1$ (left panel in Figure~\ref{fig:review}) along the trivial branch there is a last bifurcation point which gives rise to a single bright soliton branch (red line). This branch has a turning point, at which the solutions change from unstable (dashed) to stable (solid), and after the turning point it tends back towards the trivial branch. Thus, the red line in the left panel of Figure~\ref{fig:review} represents two different but almost identical curves, which can be seen in the enlarged inset. We have chosen a solution at the point $BP$ on the stable branch as a starting point for the illustration of Theorems~\ref{Fortsetzung_nichttrivial} and \ref{thm:spectral_stability}.

In the case where $d=-0.1$ (right panel in Figure~\ref{fig:review}) along the trivial branch there is a first bifurcation point from which a single dark soliton branch (red line) bifurcates. Near the second turning point of this branch the most localized single solitons live and we have chosen a stable dark soliton solution at the point $BP$ as a starting point for the illustration of  Theorems~\ref{Fortsetzung_nichttrivial} and \ref{thm:spectral_stability}.

Next we explain the global picture in Figure~\ref{fig:Global_Bif_diag} of the continuation in $\epsilon$ of the chosen point BPs from the $\epsilon=0$ case in Figure~\ref{fig:review}. The local picture is covered by Theorem~\ref{Fortsetzung_nichttrivial}. First we note the following symmetry: since $V(x)$ is even around $x=0$ we find that $(u(x),\epsilon)$ solves \eqref{TWE} if and only if $(u(-x),-\epsilon)$ satisfies \eqref{TWE}. Since reflecting $u$ does not affect the $L^2$-norm we see for $\epsilon>0$ an  exact mirror image of the one for $\epsilon<0$. 

% Bifurcation Plot d>0 und d<0

\begin{figure*}[ht]
\centering
\begin{minipage}[t]{0.32\textwidth}
\begin{tikzpicture}[overlay]
\node at (2.655,2.07)
{\includegraphics[width=\columnwidth]{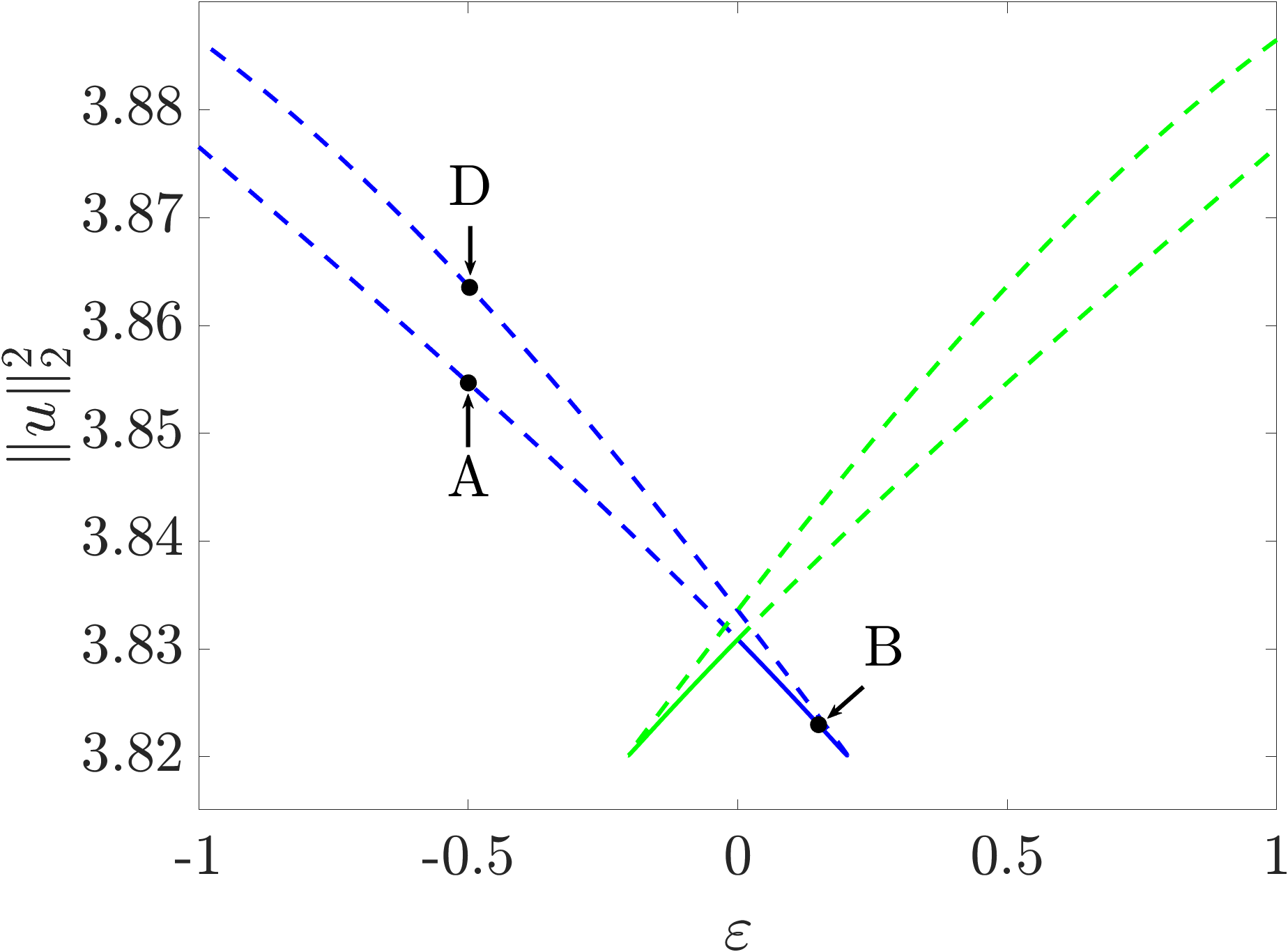}};
\node  at (3.045,3.5) {\includegraphics[width=0.3\columnwidth]{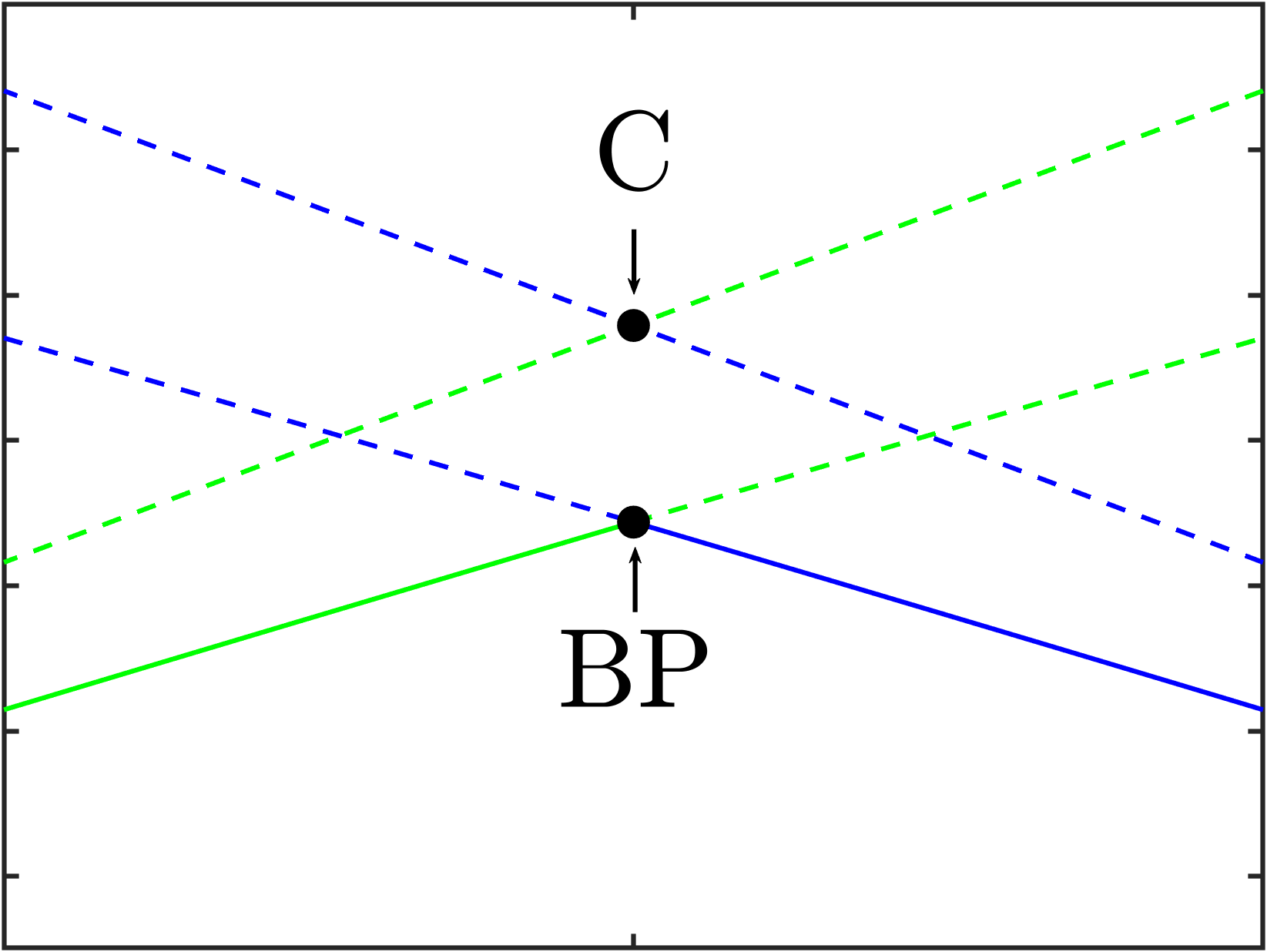}};
\draw [draw=black] (2.95,1.26) rectangle (3.15,1.7);
\draw [draw=black] (2.95,1.7) -- (2.25,2.9);
\draw [draw=black] (3.15,1.7) -- (3.83,2.9);
\end{tikzpicture}
\end{minipage}\hspace*{0.15cm}
\begin{minipage}[t]{0.32\textwidth}
\includegraphics[width=\columnwidth]{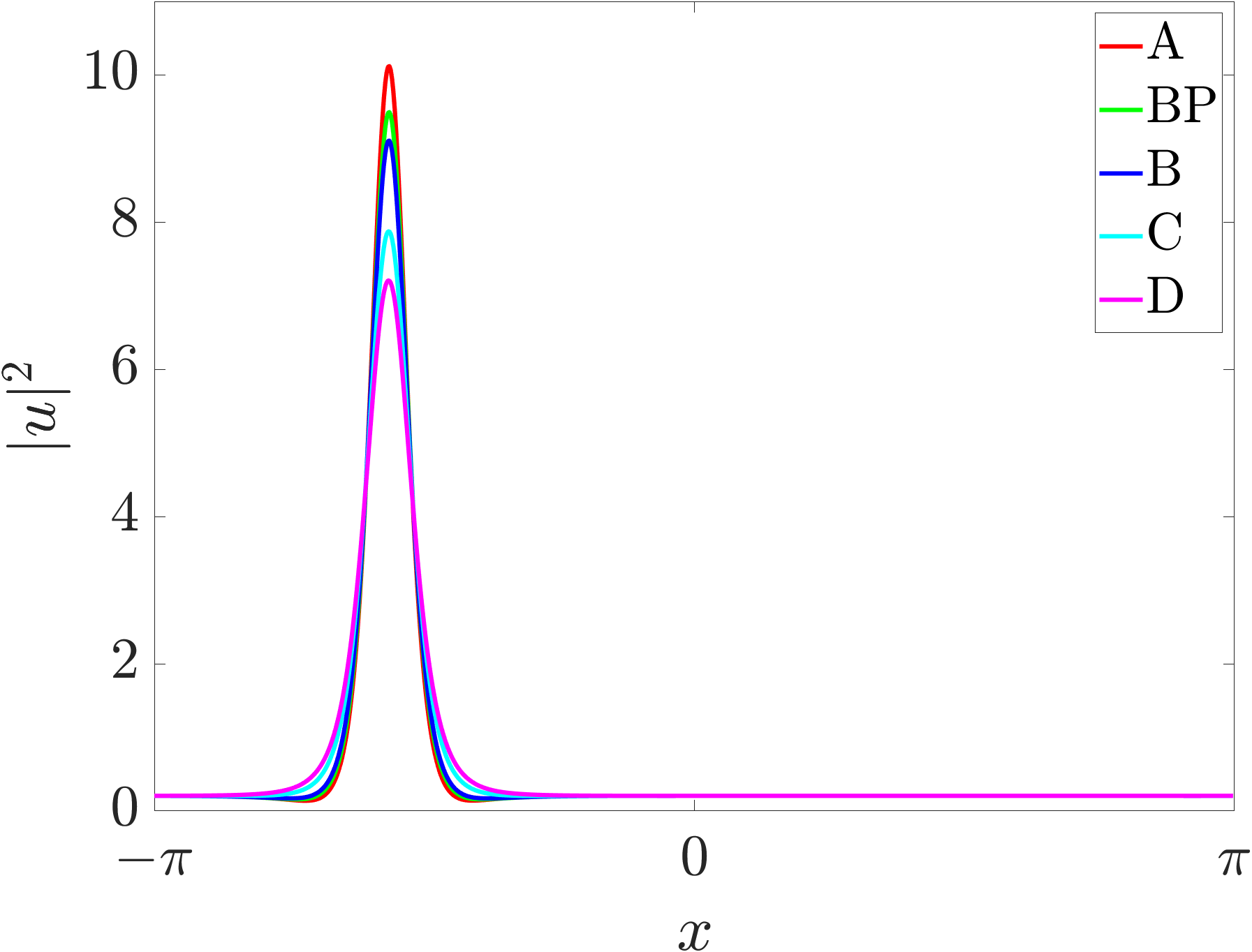}
\end{minipage} \\
\begin{minipage}[h]{0.3\textwidth}
\includegraphics[width=\columnwidth]{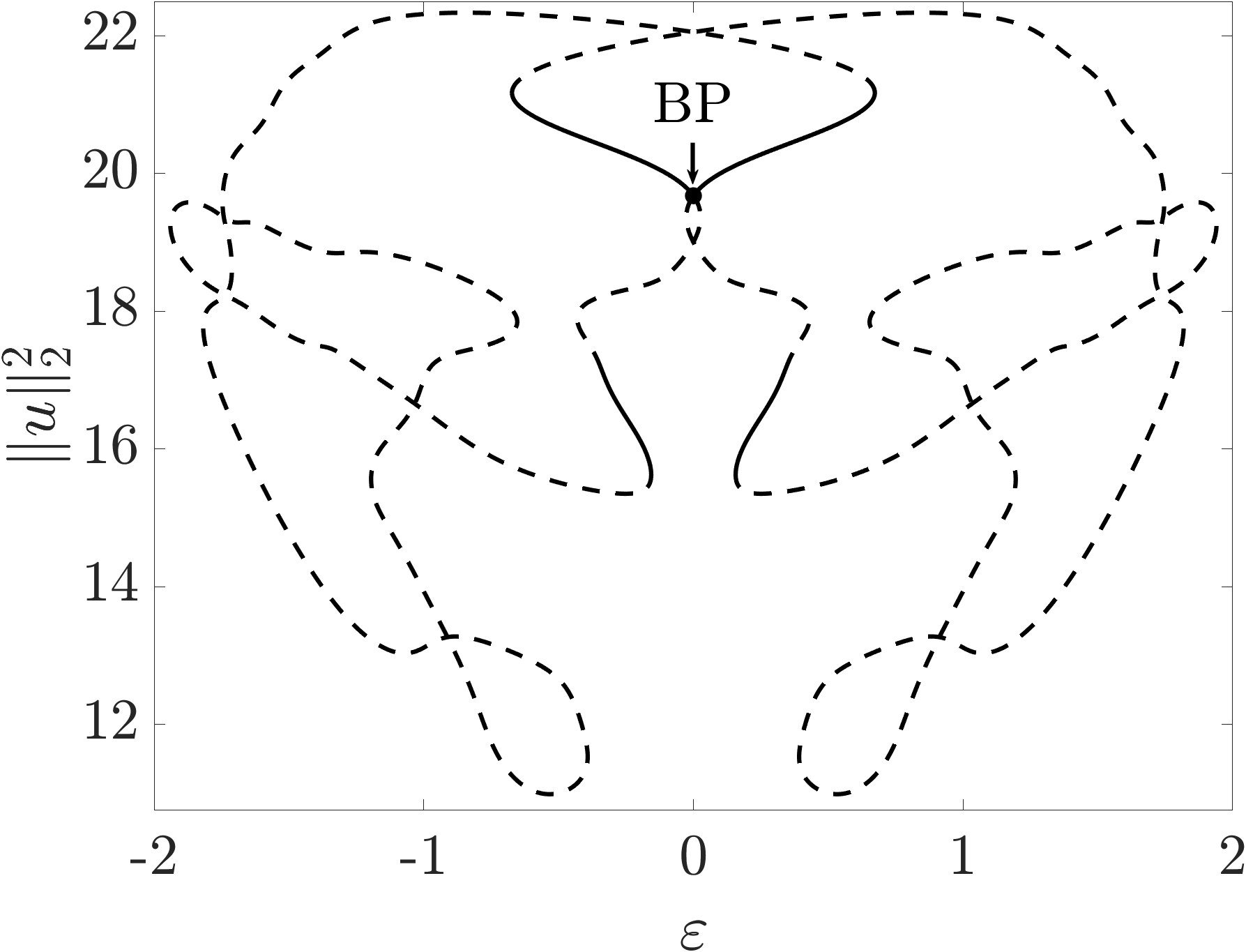} %\\[0.5cm]
\end{minipage}
\hspace*{0.15cm}
\begin{minipage}[h]{0.3\textwidth}
\includegraphics[width=\columnwidth]{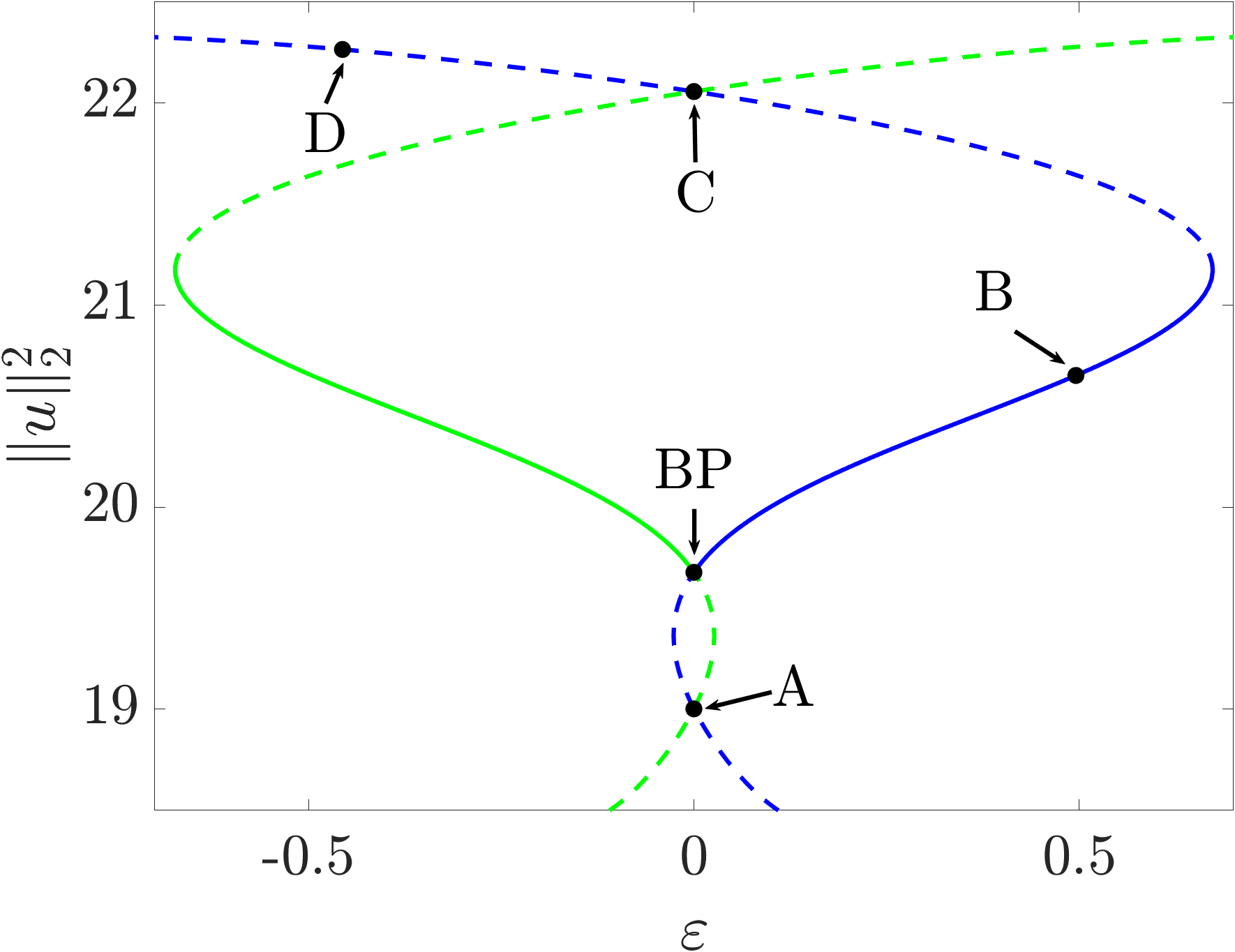} %\\[0.5cm]
\end{minipage}
\hspace*{0.15cm}
\begin{minipage}[h]{0.3\textwidth}
\includegraphics[width=\columnwidth]{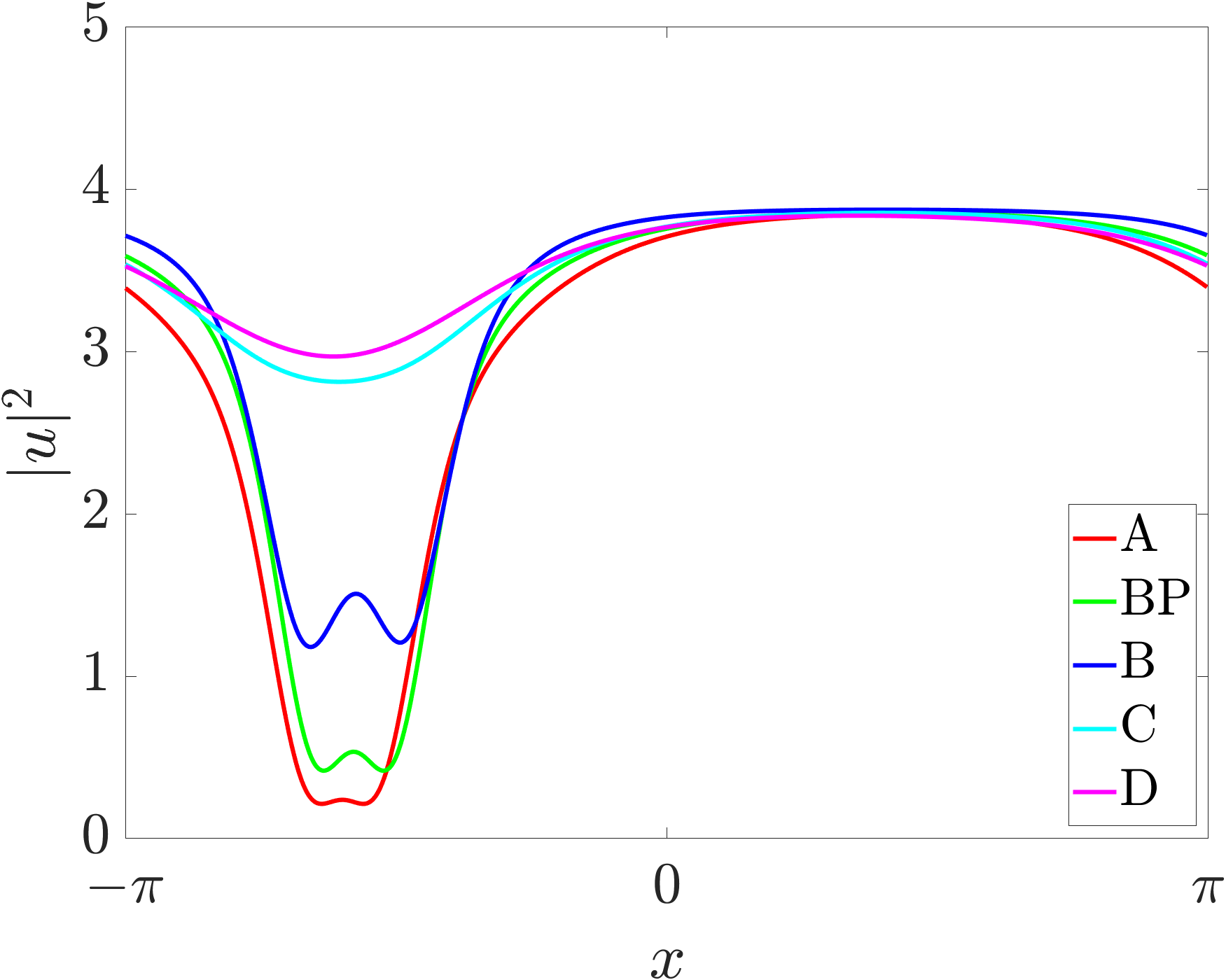} %\\[0.5cm]
\end{minipage}
\caption{Continuation diagrams w.r.t $\epsilon$ with stability regions (solid = stable; dashed = unstable) and solutions at designated points. The two different zeroes of $V_\text{eff}$ give rise to two different continuation curves (blue and green). Top panels: $d=0.1$, $\zeta=3.7$. Bottom panels: $d=-0.1$, $\zeta=4.5$ with zoom (middle panel) of the continuation curve near the starting point. }
\label{fig:Global_Bif_diag}
\end{figure*}
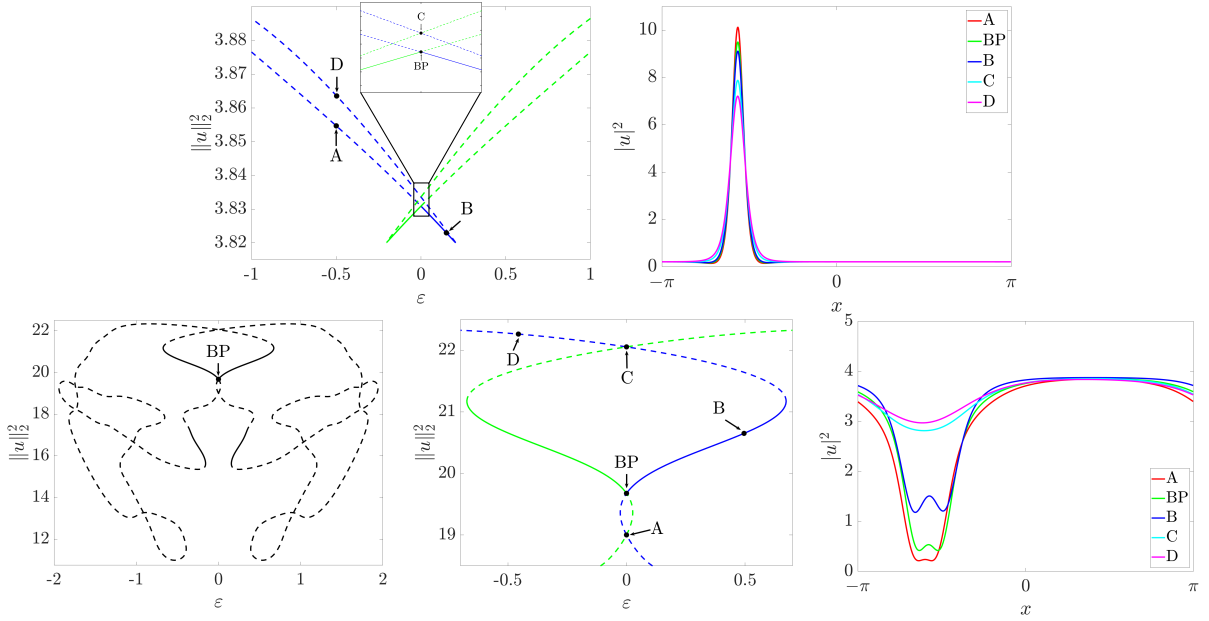

Next we observe that continuation curves in $\epsilon$ appear to be unbounded for $d=0.1$ (upper left panel of Figure~\ref{fig:Global_Bif_diag}) and closed and bounded for $d=-0.1$ (lower left panel of Figure~\ref{fig:Global_Bif_diag}). In our example the map $\sigma\mapsto V_\text{eff}(\sigma):= \RT \int_{-\pi}^{\pi} \iu V(x+\sigma) u_0'\bar\phi_0^*\,dx$ has two zeroes in the periodic domain $\mathbb{T}$ denoted by $\sigma_0$ and $\sigma_1$. Since moreover $u_0$ is even and consequently $u_0'$, $\phi_0^*$ are odd we see that the effective potential $V_\text{eff}$ is also even and hence $\sigma_0=-\sigma_1$. Thus, continuation in $\epsilon$ works for the starting point $u_0(\cdot-\sigma_0)$ (blue curve) and $u_0(\cdot+\sigma_0)$ (green curve) with $\sigma_0 <  0$. As predicted from Theorem~\ref{thm:spectral_stability} locally on one side of $\epsilon=0$ we have stable and on the other side unstable solutions. On the top and bottom right panels of Figure~\ref{fig:Global_Bif_diag} we see the graph of $|u|^2$ for several solutions on the continuation diagram. The top left panel and the bottem left panel indicate that the $\epsilon$-continuation curves meet all other nontrivial points (C for $d=0.1$ and A, C for $d=-0.1$) at $\epsilon=0$ from Figure~\ref{fig:review}.

% Effective Potential and shifts of Solitons

In Figure~\ref{fig:Effective_potential} we show the starting solutions $u_0(x-\sigma_0)$ and $u_0(x-\sigma_1)$ together with the potential $V(x)$. Here the zeroes $\sigma_0<0<\sigma_1$ of the effective potential $V_\text{eff}$ are shown as blue and green dots and we already observed $\sigma_0=-\sigma_1$ due to the evenness of both $V$ and $V_\text{eff}$. Since $u_0$ is sufficiently strongly localized the zeroes of $V_\text{eff}$ are well approximated by the zeroes of $V$ and the starting solutions are thus centered near the zeroes of $V$. Therefore, by applying Remark~\ref{rem:approx_potential}, we see that slope of $V$ at the center of the soliton being positive in the blue bifurcation point indicates that the $\epsilon$-continuation will be stable for $\epsilon>0$ and unstable for $\epsilon<0$. The stability behavior is exactly opposite for the green bifurcation point. The stability considerations are valid both for $d=0.1$ and $d=-0.1$.

\begin{figure*}
\centering
%\begin{minipage}[h]{0.32\textwidth}
%\includegraphics[width=\columnwidth]{effpot_pos.png} %\\[0.5cm]
%\end{minipage}
%\hspace*{0.15cm}
\begin{minipage}[h]{0.32\textwidth}
\includegraphics[width=\columnwidth]{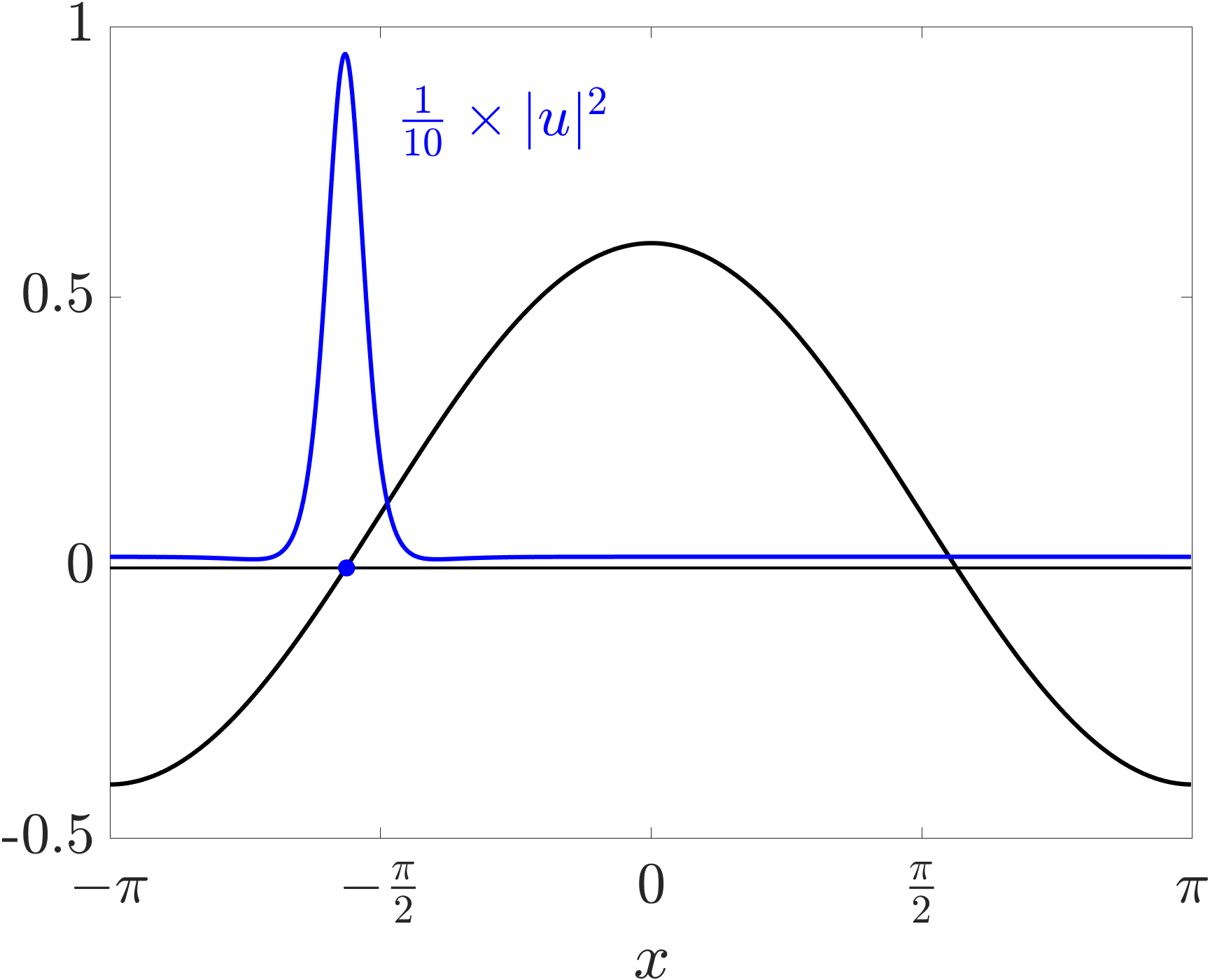} %\\[0.5cm]
\end{minipage}
\hspace*{0.15cm}
\begin{minipage}[h]{0.32\textwidth}
\includegraphics[width=\columnwidth]{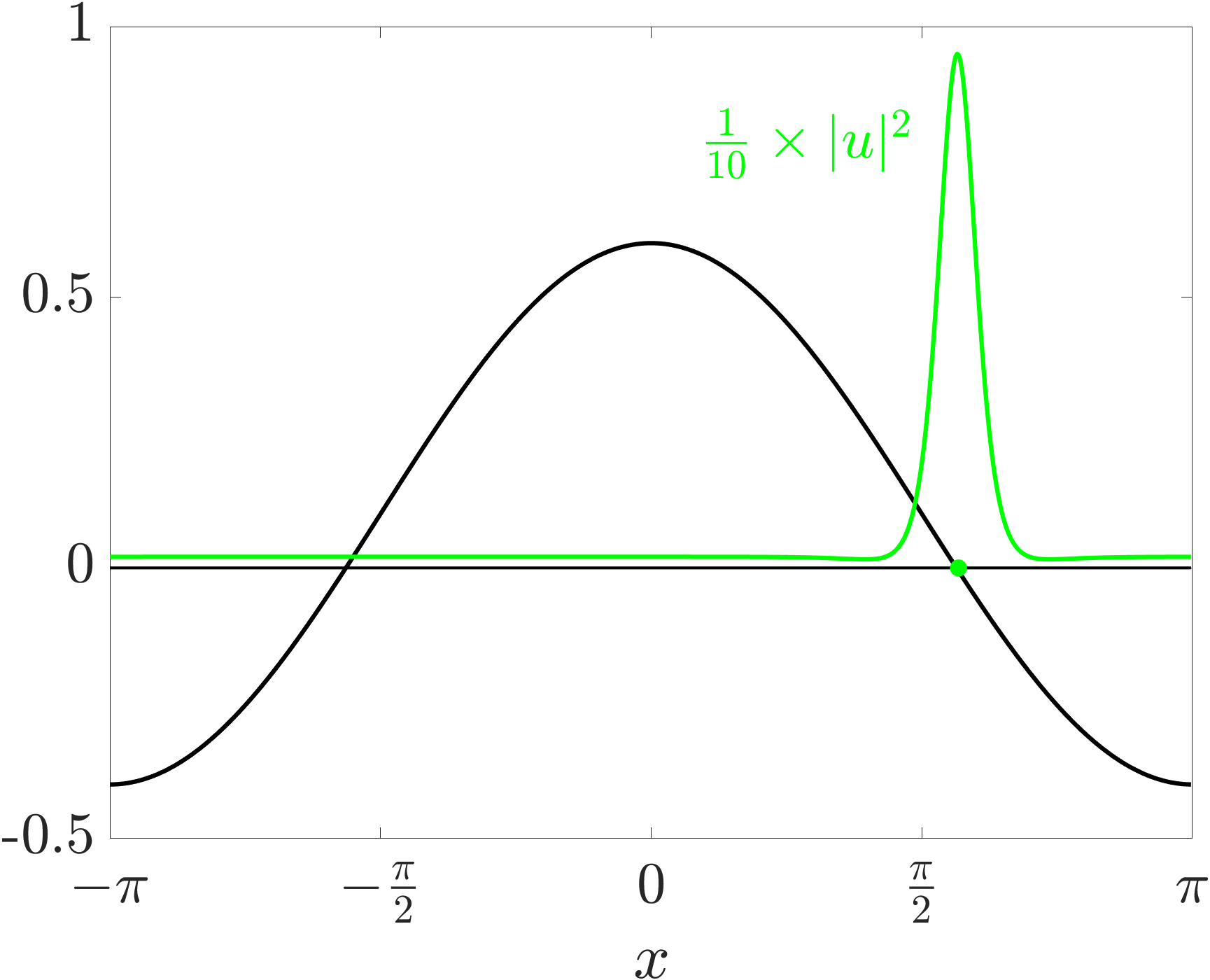} %\\[0.5cm]
\end{minipage}\\
%\begin{minipage}[h]{0.32\textwidth}
%\includegraphics[width=\columnwidth]{effpot_neg.png} %\\[0.5cm]
%\end{minipage}
%\hspace*{0.15cm}
\begin{minipage}[h]{0.32\textwidth}
\includegraphics[width=\columnwidth]{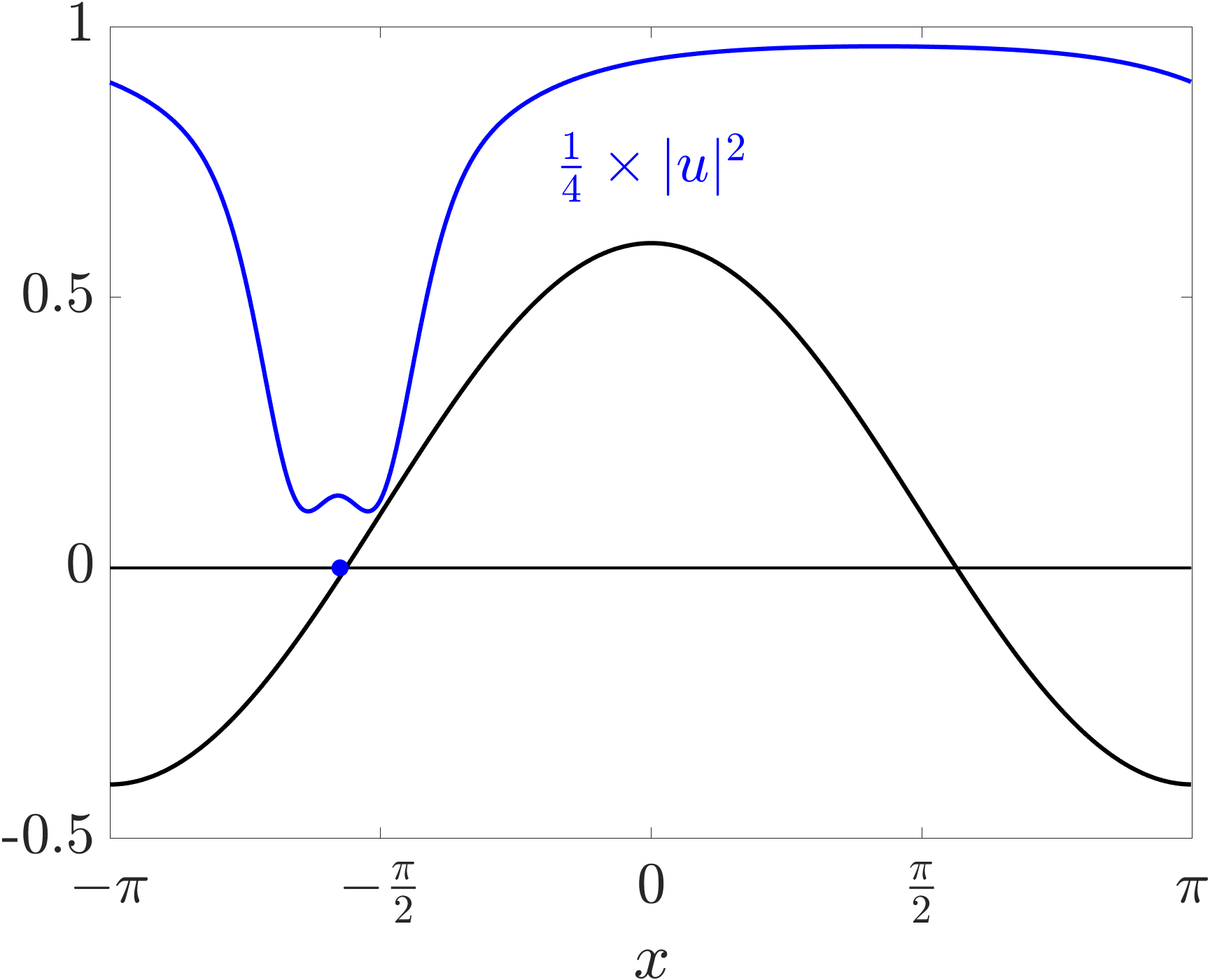} %\\[0.5cm]
\end{minipage}
\hspace*{0.15cm}
\begin{minipage}[h]{0.32\textwidth}
\includegraphics[width=\columnwidth]{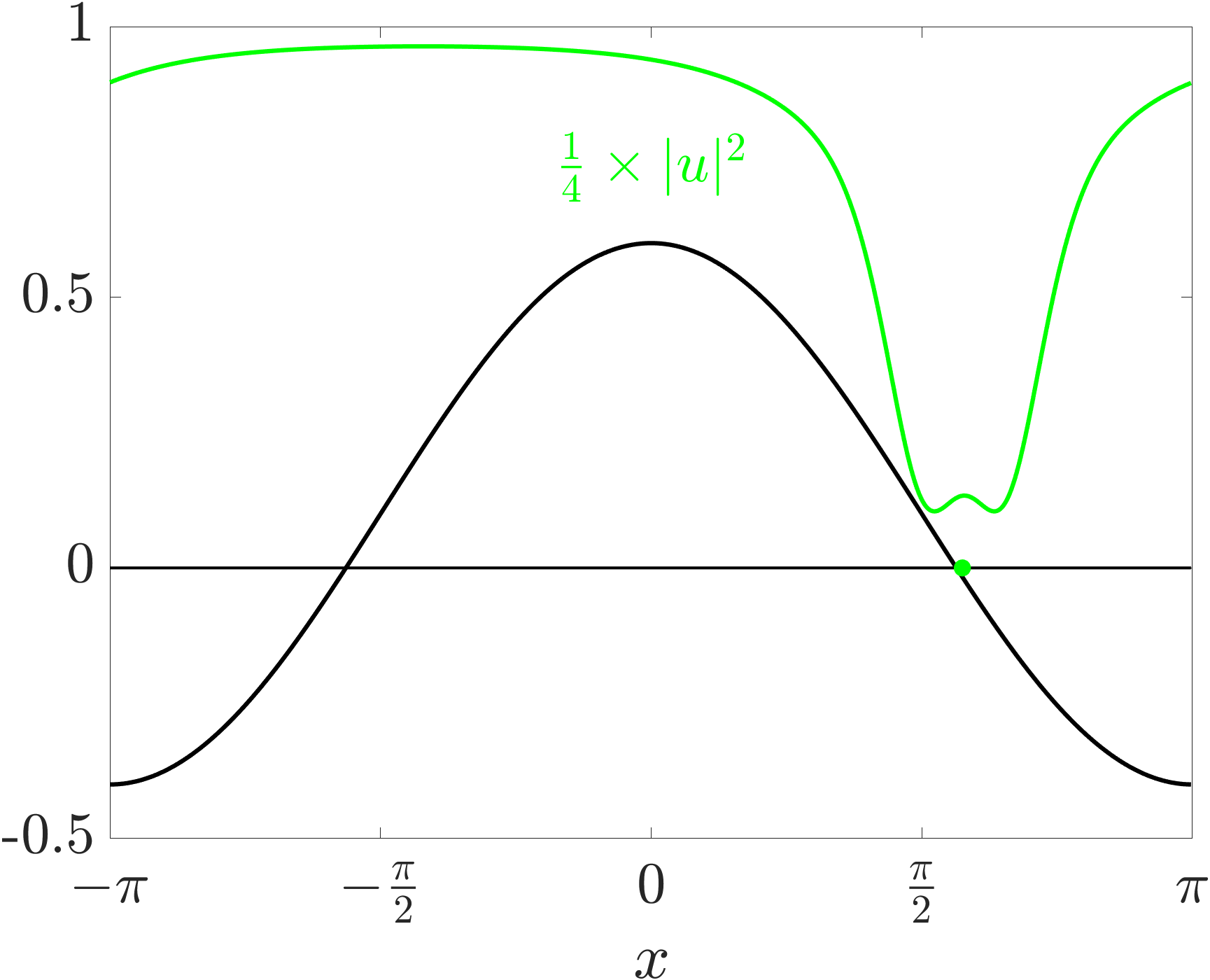} %\\[0.5cm]
\end{minipage}
\caption{Top row: $d=0.1$, bottom row: $d=-0.1$. Left panels: starting solutions $u_0(x-\sigma_0)$ together with $V(x)$ and negative zero $\sigma_0$ of $V_\text{eff}$ (blue dot). Stability for $\epsilon>0$, instability for $\epsilon<0$. Right panels: starting solutions $u_0(x+\sigma_0)$ together with $V(x)$ and positive zero $\sigma_1=-\sigma_0$ of $V_\text{eff}$ (green dot). Stability for $\epsilon<0$, instability for $\epsilon>0$.}
\label{fig:Effective_potential}
\end{figure*}

% \begin{figure*}
% \centering
% \begin{minipage}[h]{0.3\textwidth}
% \includegraphics[width=\columnwidth]{effpot_neg.png} %\\[0.5cm]
% \end{minipage}
% \hspace*{0.15cm}
% \begin{minipage}[h]{0.3\textwidth}
% \includegraphics[width=\columnwidth]{shift_neg_blau.png} %\\[0.5cm]
% \end{minipage}
% \hspace*{0.15cm}
% \begin{minipage}[h]{0.3\textwidth}
% \includegraphics[width=\columnwidth]{shift_neg_green.png} %\\[0.5cm]
% \end{minipage}
% \caption{Left: Effective potential. Middle and right: Shifted and scaled soliton solution and $V$. $d<0$.}
% \label{fig:Effective_potential_neg}
% \end{figure*}

%%% Plots of spectra
% \begin{figure*}
% \centering
% \begin{minipage}[h]{0.3\textwidth}
% \includegraphics[width=\columnwidth]{spec_pos_os.png} %\\[0.5cm]
% \end{minipage}
% \hspace*{0.15cm}
% \begin{minipage}[h]{0.3\textwidth}
% \includegraphics[width=\columnwidth]{spec_pos_stab.png} %\\[0.5cm]
% \end{minipage}
% \hspace*{0.15cm}
% \begin{minipage}[h]{0.3\textwidth}
% \includegraphics[width=\columnwidth]{spec_pos_unstab.png} %\\[0.5cm]
% \end{minipage}
% \caption{Marginally stable, stable, and unstable spectrum. $d>0$.}
% \label{fig:Spectra_pos}
% \end{figure*}

Finally, let us illustrate the spectral stability properties of the $\epsilon$-continuations in Figure~\ref{fig:Spectra}. For $\epsilon=0$ we see in the left panel the spectrum of the linearization around $u_0$ with most of spectrum having real part $-1$ due to damping $\mu=1$ and further spectrum in the left half plane together with the zero eigenvalue caused by shift-invariance. Now we consider how the critical eigenvalue behaves when $\epsilon$ varies. We do this for the case where the starting soliton sits at a zero of $V_\text{eff}$ with positive slope, cf. blue bifurcation point in Figure~\ref{fig:Effective_potential}. As predicted, the critical eigenvalue moves into the complex left half plane for $\epsilon>0$ rendering the $\epsilon$-continuations stable. Since the starting solitons are sufficiently localized $-V'(\sigma_0)$ predicts well the slope of the critical eigenvalue, cf. Lemma~\ref{lem:eig_0} and Remark~\ref{rem:approx_potential}.

\begin{figure*}
\centering
\begin{minipage}[h]{0.32\textwidth}
\includegraphics[width=\columnwidth]{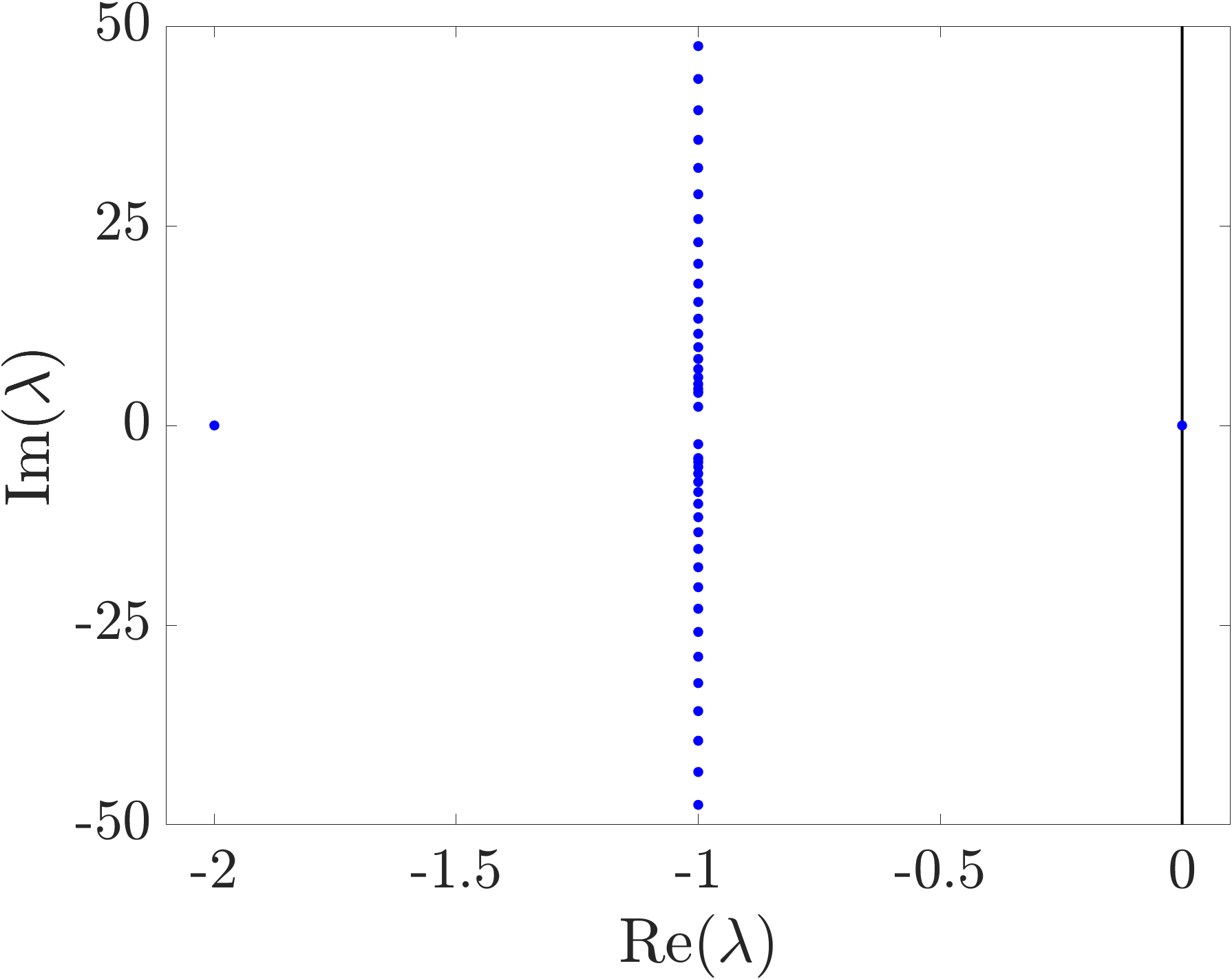} %\\[0.5cm]
\end{minipage}
\hspace*{0.15cm}
\begin{minipage}[h]{0.32\textwidth}
\includegraphics[width=\columnwidth]{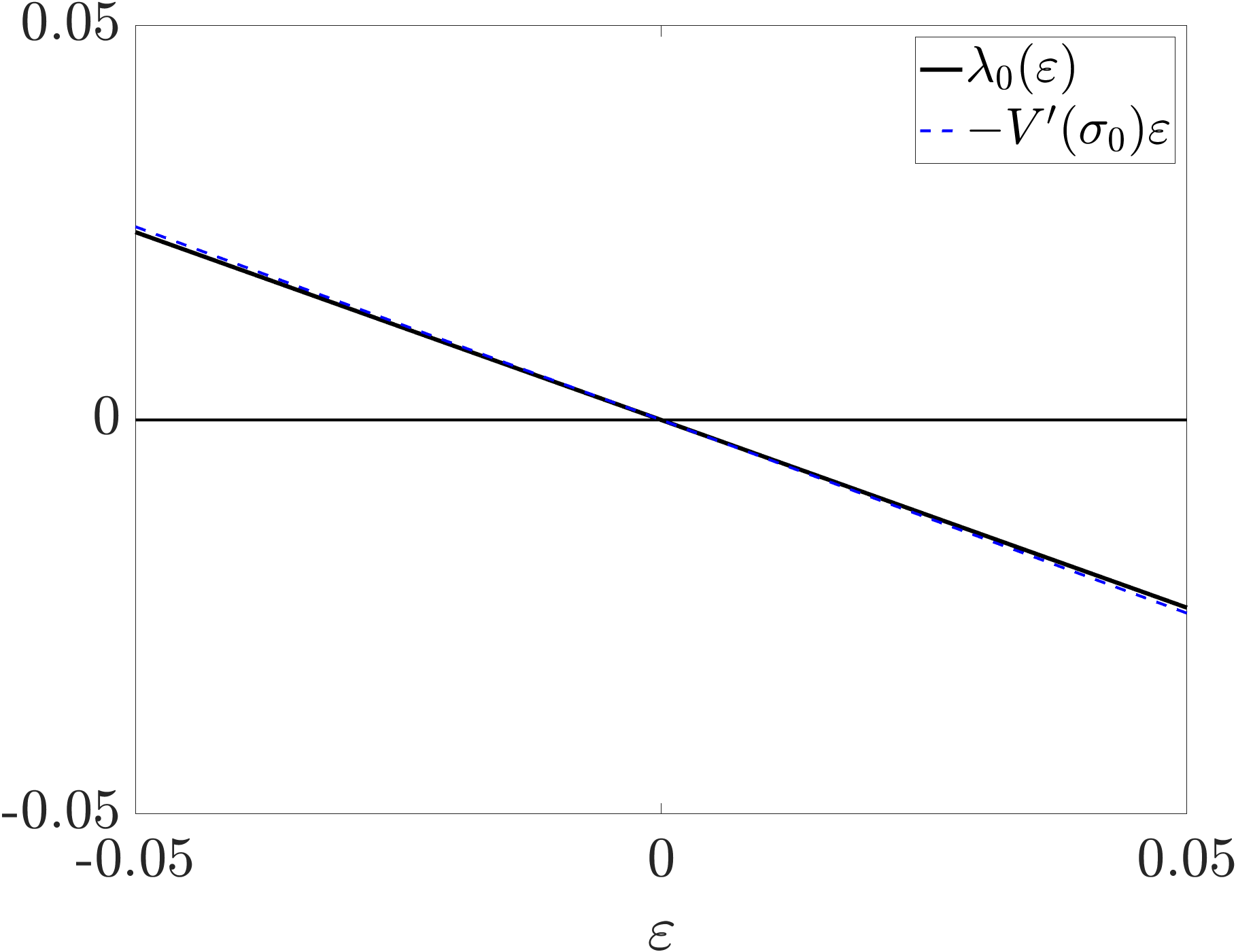} %\\[0.5cm]
\end{minipage}\\
\begin{minipage}[h]{0.32\textwidth}
\includegraphics[width=\columnwidth]{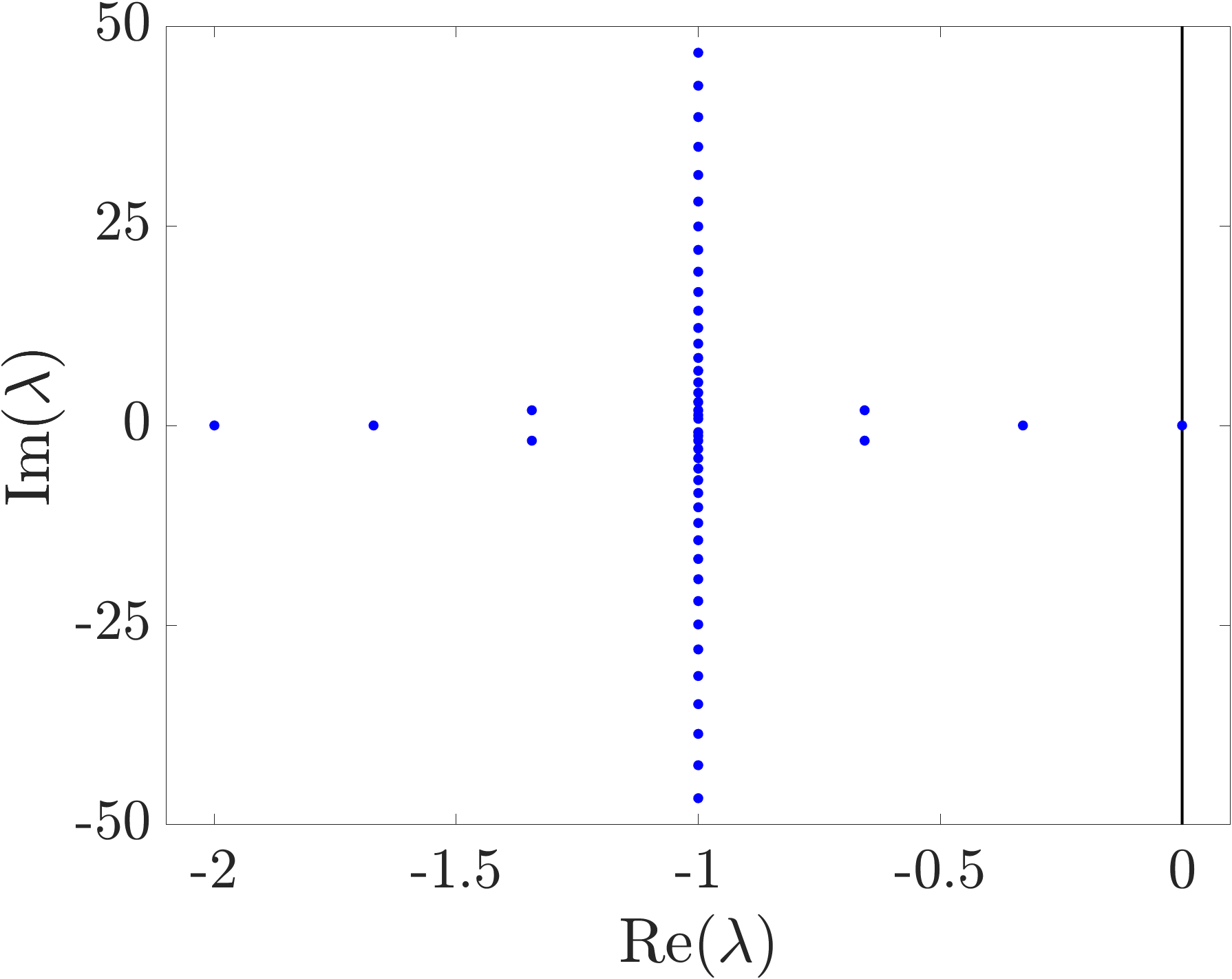} %\\[0.5cm]
\end{minipage}
\hspace*{0.15cm}
\begin{minipage}[h]{0.32\textwidth}
\includegraphics[width=\columnwidth]{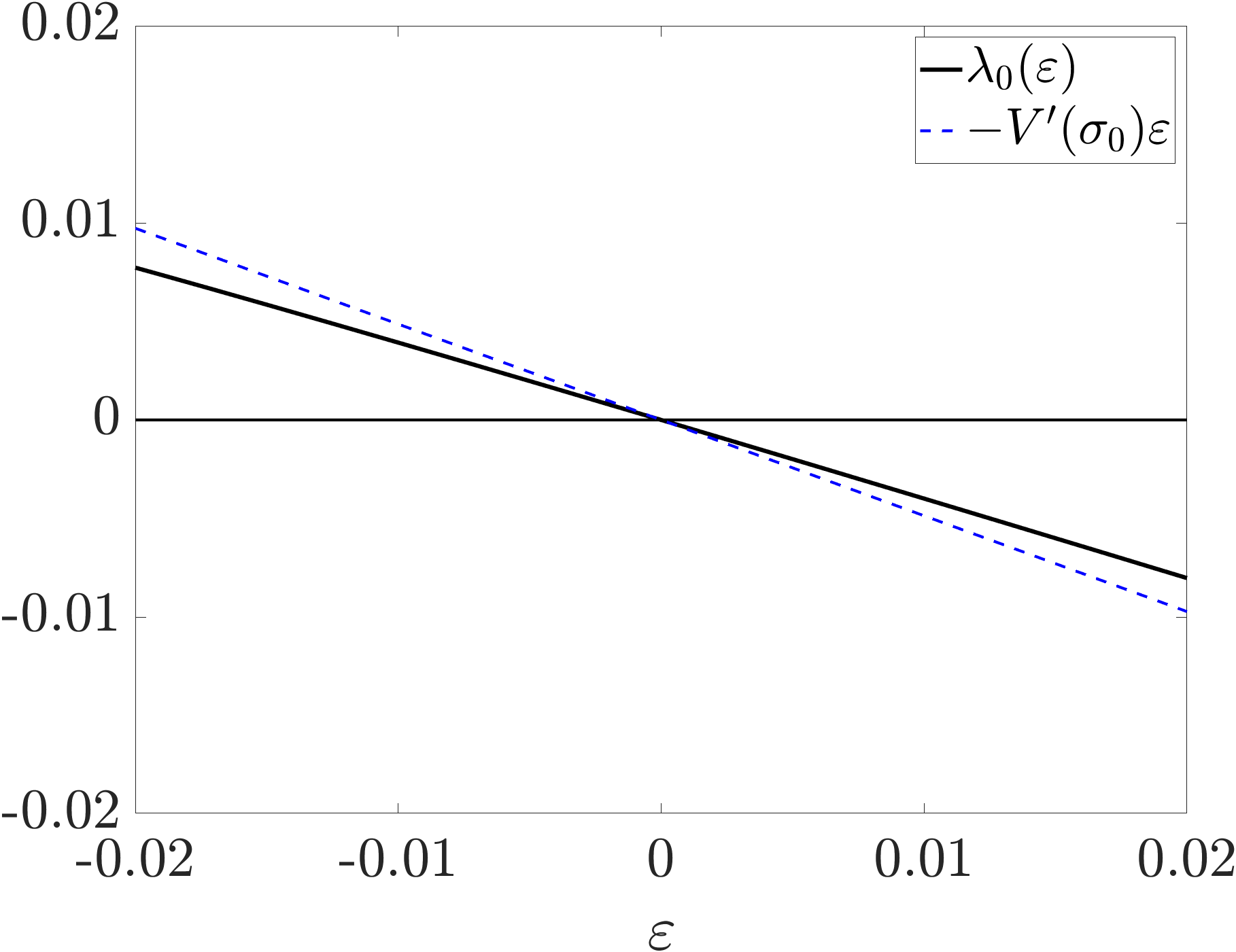} %\\[0.5cm]
\end{minipage}
\caption{Top: $d=0.1$, bottom: $d=-0.1$. Left: spectrum for $\epsilon=0$. Right: critical eigenvalue $\lambda_0(\epsilon)$ together with $-V'(\sigma_0)\epsilon$ as functions of $\epsilon$.}
\label{fig:Spectra}
\end{figure*}

\section{Proof of the existence result}

Theorem~\ref{Fortsetzung_nichttrivial} will be proved via Lyapunov-Schmidt reduction and the Implicit Function Theorem. Fix the values of $d, \zeta,\mu$ and $f_0$. Let $u_0\in H^2_\text{per}([-\pi,\pi],\C)$ be a non-degenerate solution of \eqref{TWE} for $\epsilon=0$ and recall that for $\sigma\in\R$ its shifted copy $u_\sigma(x):=u_0(x-\sigma)$ is also a solution of \eqref{TWE} for $\epsilon=0$. 

\begin{proof}[Proof of Theorem \ref{Fortsetzung_nichttrivial}:]
We seek solutions $u$ of \eqref{TWE} of the form
$$
u = u_\sigma + v,\quad \langle v,u_{\sigma}' \rangle_{L^2} = 0, \quad v \in H_\per^2([-\pi,\pi],\C).
$$
Inserting it into \eqref{TWE} we obtain the following equation for the correction term $v$:
\begin{align}\label{eq:ansatz}
L_{u_\sigma} v + \iu \epsilon V(u_\sigma' + v') - N(v,\sigma) = 0
\end{align}
with nonlinearity given by
\begin{align*}
N(v,\sigma) = \bar{u}_\sigma  v^2 + 2 u_\sigma |v|^2 + |v|^2 v.
\end{align*}
The nonlinearity is a sum of quadratic and cubic terms in $v$. Since $H^2_{\rm per}$ is a Banach algebra, it is clear that 
	for every $R > 0$, there exists $C_R > 0$ such that 
\begin{equation}
\label{bound-nonlinear}
\|N(v,\sigma)\|_{L^2} \leq C_R \|v\|_{H^2}^2, \quad \mbox{\rm for every } \; 
v \in H^2_{\rm per} : \;\; \| v \|_{H^2} \leq R.
\end{equation}
Moreover, since $V \in L^\infty$ it follows that
$$
\|\iu \epsilon V(u_\sigma'+ v')\|_{L^2} \leq  |\epsilon| \|V\|_{L^\infty} \|u_\sigma + v\|_{H^2}.
$$
Next we solve \eqref{eq:ansatz} according to the Lyapunov-Schmidt reduction method. Define the orthogonal projections
$$
P_\sigma : L^2\to \spann\{u_{\sigma}'\} \subset L^2, \quad Q_\sigma: L^2\to \spann\{\phi_{\sigma}^*\}^\perp \subset L^2
$$
onto $\ker L_{u_\sigma}$ and $(\ker L_{u_\sigma}^*)^\perp = \spann\{\phi_{\sigma}^*\}^\perp = \range L_{u_\sigma}$, respectively. Then \eqref{eq:ansatz} can be decomposed into a non-singular and singular equation
\begin{align}
Q_\sigma \left(L_{u_\sigma} (I - P_\sigma) v + \iu \epsilon V(u_{\sigma}'+v') - N(v,\sigma)\right) &= 0, \label{eq:LS1} \\
\langle \iu \epsilon V u_\sigma', \phi_{\sigma}^* \rangle_{L^2} + \langle \iu \epsilon V v' - N(v,\sigma), \phi_{\sigma}^* \rangle_{L^2} &= 0. \label{eq:LS2}
\end{align}
Notice that the linear part $Q_\sigma L_{u_\sigma} (I - P_\sigma)$ in \eqref{eq:LS1} is invertible between the $\sigma$-dependent subspaces $(\ker L_{u_\sigma})^\perp$ and $\range L_{u_\sigma}$. Therefore, the Implicit Function Theorem cannot be applied directly to solve \eqref{eq:LS1}. However, \eqref{eq:LS1} is equivalent to $F(v,\epsilon,\sigma)=0$ with
$$
F(v,\epsilon,\sigma):= Q_\sigma \left(L_{u_\sigma} (I - P_\sigma) v + \iu \epsilon V(u_{\sigma}'+v') - N(v,\sigma)\right) + \phi_\sigma^* \langle v , u_\sigma' \rangle_{L^2}
$$
and $F:H_\per^2([-\pi,\pi],\C) \times \R\times \R \to L^2([-\pi,\pi],\C)$. Here the added term $\phi_\sigma^* \langle v , u_\sigma' \rangle_{L^2}$ enforces $v\perp u_\sigma'$. For any fixed $\sigma_0 \in \R$ we have $F(0,0,\sigma_0) = 0$. Since
$$
D_v F(0,0,\sigma_0) \varphi =  L_{\sigma_0} \varphi + \phi_{\sigma_0}^* \langle \varphi , u_{\sigma_0}' \rangle_{L^2}
$$
is an isomorphism from $H_\per^2$ to $L^2$, we can apply the Implicit Function Theorem to the function $F$ which gives the existence of a smooth function $v = v(\epsilon,\sigma)$ solving the problem $F(v(\epsilon,\sigma),\epsilon,\sigma) = 0$ for $(\epsilon,\sigma)$ in a neighborhood of $(0,\sigma_0)$. Then, by construction, $v$ is a solution of \eqref{eq:LS1} and satisfies the orthogonality condition
$$
\langle v(\epsilon,\sigma) , u_\sigma' \rangle_{L^2} = 0
$$
as required at the beginning of the proof. Moreover from \eqref{eq:LS1} we see that $F(0,0,\sigma)=0$ so that $v(0,\sigma)=0$ which implies the bound
\begin{align}\label{eq:bounds_epsilon}
\|v(\epsilon,\sigma)\|_{H^2} \leq C |\epsilon|.
\end{align}
As a consequence, $\|v'(\epsilon,\sigma)\|_{L^2} \leq  C |\epsilon|$, where $v'(\epsilon,\sigma)$ denotes the derivative of $v$ with respect to $x$. Inserting $v(\epsilon,\sigma)$ into the singular equation \eqref{eq:LS2} we end up with with the 2-dimensional problem
$$
f(\epsilon,\sigma) :=  \langle \iu \epsilon V u_\sigma', \phi_{\sigma}^* \rangle_{L^2} + \langle \iu \epsilon V v'(\epsilon,\sigma) - N(v(\epsilon,\sigma),\sigma), \phi_{\sigma}^* \rangle_{L^2} = 0.
$$
For all $\sigma \in \R$ we have the asymptotic
$$
|\langle  \iu \epsilon V v'(\epsilon,\sigma) - N(v(\epsilon,\sigma),\sigma), \phi_{\sigma}^* \rangle_{L^2}| = \mathcal{O}(\epsilon^2) \;\; \text{ as } \; \epsilon \to 0
$$ 
which follows from the bounds \eqref{bound-nonlinear} and \eqref{eq:bounds_epsilon}. Thus $f$ can be written as
$$
f(\epsilon,\sigma) = \epsilon \langle \iu   V u_\sigma', \phi_{\sigma}^* \rangle_{L^2} + \mathcal{O}(\epsilon^2) \;\; \text{ as } \;\; \epsilon \to 0.
$$
Note that if $\langle \iu   V u_\sigma', \phi_{\sigma}^* \rangle_{L^2} \not= 0$ the function $f(\epsilon,\sigma)$ has no root near $(0,\sigma)$ other than the trivial root $(0,\sigma)$. However, by our assumption on the effective potential $V_{\text{eff}}$ there exists $\sigma_0 \in \R$ such that
$$
\langle \iu   V u_{\sigma_0}', \phi_{\sigma_0}^* \rangle_{L^2} = \RT \int_{-\pi}^\pi \iu V(x) u_{\sigma_0}' \bar{\phi}_{\sigma_0}^* dx =V_{\text{eff}}(\sigma_0) = 0
$$
and
$$
\left.\partial_\sigma\langle \iu   V u_\sigma', \phi_{\sigma}^* \rangle_{L^2} \right\vert_{\sigma=\sigma_0} = \left.\partial_\sigma \RT \int_{-\pi}^\pi \iu V(x) u_{\sigma}' \bar{\phi}_\sigma^* dx \right\vert_{\sigma=\sigma_0}
=V'_{\text{eff}}(\sigma_0) \not= 0.
$$
Hence the Implicit Function Theorem can be applied to the function $\epsilon^{-1} f(\epsilon,\sigma)$ and yields a curve of unique non-trivial solutions $\sigma = \sigma(\epsilon)$ to the singular equation $f(\epsilon,\sigma) = 0$ such that $\sigma(0) = \sigma_0$. Finally we conclude that $u(\epsilon) = u_0(\cdot-\sigma(\epsilon))+ v(\epsilon,\sigma(\epsilon))$ solves \eqref{TWE} for small $\epsilon$.
\end{proof}

\section{Proof of the stability result}

In this section we will find the condition when the stationary solutions obtained in Theorem~\ref{Fortsetzung_nichttrivial} as a continuation of a stable solution $u_0$ of the LLE \eqref{LLE_original} 
are spectrally stable against co-periodic perturbations in the perturbed LLE (\ref{TWE_dyn}). Moreover, we prove the nonlinear asymptotic stability of stationary spectrally stable solutions. 

\subsection{Preliminary notes}

For our stability analysis we consider \eqref{TWE_dyn} as a 2 dimensional system by decomposing the function $u = u_1 + \iu u_2$ into real and imaginary part. This leads us to the system of dynamical equations
\begin{align}\label{eq:2d_LLE}
\left\{
\begin{array}{l}
\partial_t u_1  =  -d \partial_{x}^2 u_2 + \epsilon V(x) \partial_x u_1 + \zeta u_2 - \mu u_1 - (u_1^2+u_2^2) u_2 + f_0, \\ 
\partial_t u_2  =  d \partial_{x}^2 u_1 + \epsilon V(x) \partial_x u_2 - \zeta u_1 - \mu u_2 + (u_1^2+u_2^2) u_1,
\end{array}\right.
\end{align} 
equipped with the $2\pi$-periodic boundary condition on $\mathbb{R}$. 
The spectral problem associated to the nonlinear system (\ref{eq:2d_LLE}) 
can be written as
$$
\widetilde{L}_{u,\epsilon} \bm{v} = \lambda \bm{v}, \quad\lambda \in \C, \quad \bm{v} \in H_\per^2([-\pi,\pi],\C) \times H_\per^2([-\pi,\pi],\C)
$$
and the linearized operator $\widetilde{L}_{u,\epsilon}$ is given by (\ref{decomposition}). Note that the operator $A_u$ in the decomposition  
(\ref{decomposition}) is self-adjoint on $L^2([-\pi,\pi],\C) \times L^2([-\pi,\pi],\C)$ and $\widetilde{L}_{u,\epsilon}$ is an index $0$ Fredholm operator. Moreover we see that if $u_0$ is a non-degenerate solution of \eqref{TWE} for $\epsilon = 0$ then the following relations for the linearized operators are true:
$$
\ker \widetilde{L}_{u_0,0} = \spann\{\bm{u}_0'\}, \quad \ker  \widetilde{L}_{u_0,0}^* = \spann\{J \bm{\phi}_0^*\},
$$
where the vectors $\bm{u}_0' = (u_{01}',u_{02}')$ and $\bm{\phi}_0^* = (\phi_{01}^*,  \phi_{02}^*)$ are obtained from $u_0'=u_{01}' + \iu u_{02}'$ and $\phi_0^* = \phi_{01}^* + \iu \phi_{02}^*$. We recall that $\langle \bm u_0', J\bm\phi_0^* \rangle_{L^2} = 1$ due to normalization, cf.  Remark \ref{rem-kernel}.

Finally we observe that since the embedding 
$$
H_\per^2([-\pi,\pi],\C) \times H_\per^2([-\pi,\pi],\C) \hookrightarrow L^2([-\pi,\pi],\C) \times L^2([-\pi,\pi],\C)
$$ 
is compact, the linearization has compact resolvents and thus the spectrum of $\widetilde{L}_{u,\epsilon}$ consists of isolated eigenvalues with finite multiplicity where the only possible accumulation point is at $\infty$. In the following we will use the spaces
\begin{align*}
H_\per^2([-\pi,\pi],\C) =: X, \quad
H_\per^1([-\pi,\pi],\C) =:  Y,\quad
L^2([-\pi,\pi],\C) =: Z.
\end{align*}
Both the proof of Theorem~\ref{thm:spectral_stability} and Theorem~\ref{thm:nonlinear_stability} rely on the next lemma for the linearized operator $\widetilde{L}_{u(\epsilon),\epsilon}$ where $u(\epsilon)$  lies on the solution branch of Theorem~\ref{Fortsetzung_nichttrivial} and $|\epsilon|$ is small. The lemma gives spectral bounds for eigenvalues with large imaginary part together with a uniform resolvent estimate. The proof is presented in Section~\ref{sec:proof_resolvent_est}.

\begin{Lemma}\label{lem:resolvent_est}\label{lem:resolvent_estimate}
Denote $\Lambda_{\lambda^*}:=\{\lambda \in \C: \RT(\lambda) \geq 0 , |\IT(\lambda)|\geq \lambda^* \}$. Given $\epsilon_1>0$ sufficiently small there exists $\lambda^* >0$ such that we have the uniform resolvent bound
$$
\sup_{\lambda \in \Lambda_{\lambda^*}} \|(\lambda I - \widetilde{L}_{u(\epsilon),\epsilon})^{-1}\|_{L^2 \to L^2} < \infty
$$
for all $\epsilon\in [-\epsilon_1,\epsilon_1]$. 
\end{Lemma}

\begin{Remark} \label{rem:extension_to_left} The uniformity of the resolvent estimate on the imaginary axis allows to sharpen the above result as follows. If we define $S$ as the supremum from Lemma~\ref{lem:resolvent_estimate} and let $0<\delta<1/S$ then the estimate
$$
\sup_{\lambda \in \Lambda_{\lambda^*}-\delta} \|(\lambda I - \widetilde{L}_{u(\epsilon),\epsilon})^{-1}\|_{L^2 \to L^2} < \infty
$$
holds. This follows from taking inverses in the identity
$$
(\lambda-\delta-\widetilde{L}_{u(\epsilon),\epsilon}) = (\lambda- \widetilde{L}_{u(\epsilon),\epsilon})(I-\delta(\lambda-\widetilde{L}_{u(\epsilon),\epsilon})^{-1}).
$$
\end{Remark}

\subsection{Proof of Theorem~\ref{thm:spectral_stability}}

For $\lambda \in \C$ we study the spectral problem
\begin{equation}\label{eq:spectral_problem}
\widetilde{L}_{u,\epsilon} \bm{v} = \lambda \bm{v}.
\end{equation}
Since \eqref{TWE} has the translational symmetry  in the case that $\epsilon=0$ we find
$$
\widetilde{L}_{u,0} \bm{u}'=0.
$$
For $\epsilon \not = 0$, this symmetry is broken, and the zero eigenvalue is expected to move either
into the stable or unstable half-plane. In our stability analysis, it is therefore important to understand how the critical zero eigenvalue behaves along the bifurcating solution branch given by $(-\epsilon^*,\epsilon^*) \ni \epsilon \mapsto u(\epsilon) \in X$ with $u(0) = u_{\sigma_0}$, where $\sigma_0$ is a simple zero of $V_{\mathrm{eff}}$ as in Theorem \ref{Fortsetzung_nichttrivial}. For the following calculations we will identify $u(\epsilon)$ with a vector-valued function $\bm{u}(\epsilon): \mathbb{T}\times\R\to\R^2$ and write this as $\bm{u}(\epsilon) \in X \times X$.

\medskip

We start with the tracking of the simple critical zero eigenvalue and set up the equation for the perturbed eigenvalue $\lambda_0 = \lambda_0(\epsilon)$ which reads
$$
\widetilde{L}_{u(\epsilon),\epsilon}  \bm{v}(\epsilon) = \lambda_0(\epsilon) \bm{v}(\epsilon).
$$
After a possible re-scaling we find that $\bm{v}(0) = \bm{u}_{\sigma_0}'$ and using regular perturbation theory for simple eigenvalues, cf. \cite{Kato,Kielhoefer}, the mapping $(-\epsilon^*,\epsilon^*)\ni\epsilon \mapsto \lambda_0(\epsilon) \in \R$ is continuously differentiable. Our first goal is to derive a formula for $\lambda_0'(0)$. If $\lambda_0'(0) > 0$ this means that the solutions $u(\epsilon)$ for $\epsilon  >0$ are spectrally unstable. In contrast, if  $\lambda_0'(0) < 0$, the solutions $u(\epsilon)$ for $\epsilon > 0$ are spectrally stable. 

\begin{Lemma}\label{lem:eig_0}
Let $\epsilon \mapsto \lambda_0(\epsilon)$ be the $C^1$ parametrization of the perturbed zero eigenvalue. Then the following formula holds true:
$$
\lambda_0'(0) = - \int_{-\pi}^{\pi} V'(x) \bm{u}_{\sigma_0}' \cdot J\bm{\phi}_{\sigma_0}^* dx .
$$
\end{Lemma}

\begin{proof}
On the one hand, if we differentiate the equation
$$
\widetilde{L}_{u(\epsilon),\epsilon}  \bm{v}(\epsilon) = \lambda_0(\epsilon) \bm{v}(\epsilon).
$$
with respect to $\epsilon$ and evaluate at $\epsilon = 0$ we find
$$
\widetilde{L}_{u_{\sigma_0},0}  \partial_\epsilon\bm{v}(0) - J N_u \bm{u}_{\sigma_0}' + V(x) \bm{u}_{\sigma_0}'' = \lambda_0'(0)  \bm{u}_{\sigma_0}',
$$
where $N_u$ is given by
$$
N_u = 2
\begin{pmatrix}
3 u_{\sigma_01} \partial_\epsilon u_1(0) + u_{\sigma_02} \partial_\epsilon u_2(0) & u_{\sigma_01}  \partial_\epsilon u_2(0) + u_{\sigma_02}  \partial_\epsilon u_1(0) \\
u_{\sigma_01}  \partial_\epsilon u_2(0) + u_{\sigma_02}  \partial_\epsilon u_1(0) & u_{\sigma_01} \partial_\epsilon u_1(0) + 3 u_{\sigma_02} \partial_\epsilon u_2(0)
\end{pmatrix}.
$$
On the other hand, if we differentiate \eqref{TWE} with respect to $\epsilon$ at $\epsilon=0$, then we obtain
$$
\widetilde{L}_{u_{\sigma_0},0}  \partial_\epsilon\bm{u}(0) + V(x) \bm{u}_{\sigma_0}'=0.
$$
If we differentiate this equation with respect to $x$ we find
$$
\widetilde{L}_{u_{\sigma_0},0}  \partial_\epsilon\bm{u}'(0) + V(x) \bm{u}_{\sigma_0}'' + V'(x) \bm{u}_{\sigma_0}' - JN_u \bm{u}_{\sigma_0}'= 0.
$$
Combining both equations yields
$$
\widetilde{L}_{u_{\sigma_0},0}  [\partial_\epsilon\bm{v}(0) -  \partial_\epsilon\bm{u}'(0)] -V'(x) \bm{u}_{\sigma_0}' = \lambda_0'(0)  \bm{u}_{\sigma_0}'
$$
and testing this equation with $J \bm{\phi}_{\sigma_0}^*\in \ker \widetilde{L}^*_{u_{\sigma_0},0}$ we obtain 
$$
-\int_{-\pi}^{\pi} V'(x) \bm{u}_{\sigma_0}'\cdot J \bm{\phi}_{\sigma_0}^* dx=
-\langle V'(x) \bm{u}_{\sigma_0}', J \bm{\phi}_{\sigma_0}^* \rangle_{L_2} = \lambda_0'(0) \langle \bm{u}_{\sigma_0}', J \bm{\phi}_{\sigma_0}^* \rangle_{L_2} = \lambda_0'(0)
$$
which finishes the proof. 
\end{proof}

By Lemma~\ref{lem:eig_0} we can control the critical part of the spectrum close to the origin along the bifurcating solution branch. In fact, using standard perturbation theory, cf. \cite{Kato}, we know that all the eigenvalues of $\widetilde{L}_{u(\epsilon),\epsilon}$ depend continuously on the parameter $\epsilon$. However, this dependence is in general not uniform w.r.t. all eigenvalues, so we have to make sure that no unstable spectrum occurs far from the origin. At this point, it is worth mentioning that we have an a-priori bound on the spectrum of the form
$$
\exists \lambda_* = \lambda_*(u(\epsilon),\epsilon)>0: \quad \lambda \in \sigma(\widetilde{L}_{u(\epsilon),\epsilon}) \implies \RT(\lambda) \leq \lambda_*.
$$
This bound follows from the Hille-Yoshida Theorem since $\widetilde{L}_{u(\epsilon),\epsilon}$ generates a $C_0$-semigroup on $Z\times Z$, cf. Lemma~\ref{lem:gen_in_L2} below. It can also be shown directly by testing the eigenvalue problem with the corresponding eigenfunction and integration by parts. As a conclusion, spectral stability holds if we can prove that there exists $\lambda^* >0$ such that
$$
\{ \lambda \in \C: 0 \leq \RT(\lambda)\leq \lambda_*, |\IT(\lambda)|\geq \lambda^* \} \subset \rho(\widetilde{L}_{u(\epsilon),\epsilon}).
$$
This relation is shown as part of Lemma~\ref{lem:resolvent_est} and it is extended to the left of the origin by the subsequent Remark~\ref{rem:extension_to_left}. Since in any rectangle 
$\{ \lambda \in \C: -M \leq \RT(\lambda)\leq \lambda_*, |\IT(\lambda)|\leq \lambda^* \}$ there are only finitely many eigenvalues of $\widetilde{L}_{u(\epsilon),\epsilon}$ and they depend (uniformly) continuoulsy on $\epsilon$, our assumption (A2) on $\widetilde{L}_{u_0,0}$ shows that none of these eigenvalues (except possibly the critical one) can move into the right half plane if $|\epsilon|$ is small. Therefore, only the movement of the critical eigenvalue determines the spectral stability and therefore Theorem~\ref{thm:spectral_stability} is true.

\subsection{Proof of Theorem~\ref{thm:nonlinear_stability}}

In order to prove nonlinear asymptotic stability of stationary solutions of \eqref{eq:2d_LLE} it is enough to show exponential stability of the semigroup of the linearization in $Y \times Y$, see e.g. \cite{Cazenave}. For the proof of Theorem~\ref{thm:nonlinear_stability} we will show the following three steps:
\begin{itemize}
\item[(i)] Prove that $\widetilde{L}_{u(\epsilon),\epsilon}$ is the generator of a $C_0$-semigroup on $Z \times Z$.
\item[(ii)] Show exponential decay of $(\eu^{\widetilde{L}_{u(\epsilon),\epsilon} t})_{t \geq 0}$ in $Z \times Z$.
\item[(iii)] Show exponential decay of $(\eu^{\widetilde{L}_{u(\epsilon),\epsilon} t})_{t \geq 0}$ in $Y \times Y$.
\end{itemize}

For step (i), we establish the generator properties of the linearization in $Z \times Z$.

\begin{Lemma}\label{lem:gen_in_L2}
The operator $\widetilde{L}_{u(\epsilon),\epsilon}$ generates a $C_0$-semigroup on $Z \times Z$.
\end{Lemma}

\begin{proof}
We split the operator into
$$
\widetilde{L}_{u(\epsilon),\epsilon} = L_1 + L_2 + L_3,
$$
where $L_1: X \times X \to Z \times Z$, 
$L_2: Y \times Y \to Z \times Z$, and 
$L_3: Z \times Z \to Z \times Z$ are defined by 
$$
L_1
\begin{pmatrix}
\varphi_1 \\ \varphi_2
\end{pmatrix} 
:=
\begin{pmatrix}
-d  \varphi_2'' - \mu \varphi_1 \\
d \varphi_1''  - \mu \varphi_2
\end{pmatrix}, 
$$
$$
L_2
\bm{\varphi}
:= \epsilon V(x) \bm{\varphi}' - \frac{|\epsilon|}{2} \|V'\|_{L^\infty} \bm{\varphi},
$$
and
$$
L_3
\begin{pmatrix}
\varphi_1 \\ \varphi_2
\end{pmatrix}
:=
\begin{pmatrix}
\frac{|\epsilon|}{2} \|V'\|_{L^\infty}-2u_1 u_2  & \zeta-(u_1^2+3u_2^2) \\
-\zeta + 3u_1^2+ u_2^2 & \frac{|\epsilon|}{2} \|V'\|_{L^\infty} +2 u_1 u_2
\end{pmatrix}
\begin{pmatrix}
\varphi_1 \\ \varphi_2
\end{pmatrix}
$$
We will show that 
\begin{itemize}
\item[(i)] $L_1$ generates a contraction semigroup.
\item[(ii)] $L_2$ is dissipative and bounded relative to $L_1$.
\item[(iii)] $L_3$ is a bounded operator on $Z \times Z$.
\end{itemize}
By using the semigroup theory, this will prove that the sum $L_1 + L_2 + L_3$ is the generator of a $C_0$-semigroup on $Z \times Z$.

\medskip

Part (i): It follows that $\RT \langle L_1 \bm{\varphi},\bm{\varphi}\rangle_{L^2} = -\mu \|\bm{\varphi}\|_{L^2}^2 \leq 0$ for every $\bm{\varphi} \in X \times X$, and $\lambda-L_1$ is invertible for every $\lambda>0$ which can be seen using Fourier transform. By the Lumer-Phillips Theorem we find that $L_1$ generates a contraction semigroup on $Z \times Z$.

\medskip

Part (ii): We have to show that 
$$
\forall \bm{\varphi} \in Y \times Y: \quad\RT\langle L_2 \bm{\varphi},\bm{\varphi} \rangle_{L^2} \leq 0
$$
and
$$
\forall a>0, \, \exists b>0:\quad  \|L_2 \bm{\varphi}\|_{L^2} \leq a \|L_1\bm{\varphi}\|_{L^2} + b \|\bm{\varphi}\|_{L^2} \quad\forall \bm{\varphi}\in X \times X.
$$
Let $\bm{\varphi} = (\varphi_1,\varphi_2) \in Y \times Y$ and observe that integration by parts yields
\begin{align*}
\RT \int_{-\pi}^{\pi} \epsilon V(x) (\varphi_1' \bar\varphi_1 +  \varphi_2' \bar\varphi_2 ) - \frac{|\epsilon|}{2} \|V'\|_{L^\infty} |\bm\varphi|^2 dx
= \int_{-\pi}^{\pi} -\frac{\epsilon}{2} V'(x) |\bm\varphi|^2  - \frac{|\epsilon|}{2} \|V'\|_{L^\infty} |\bm\varphi|^2 dx \leq 0
\end{align*}
which shows that $L_2$ is dissipative. Further, if $\bm{\varphi} \in X \times X$, then for every $a>0$ we have
\begin{align*}
\|\epsilon V \bm\varphi' - \frac{|\epsilon|}{2} \|V'\|_{L^\infty} \bm\varphi\|_{L^2} 
&\leq |\epsilon| \|V\|_{L^\infty} \|\bm\varphi'\|_{L^2} + \frac{|\epsilon|}{2} \|V'\|_{L^\infty} \|\bm\varphi\|_{L^2} \\
&\leq |\epsilon| a \|V\|_{L^\infty} \|\bm\varphi''\|_{L^2} + \frac{|\epsilon|}{4a} \|V\|_{L^\infty} \|\bm\varphi\|_{L^2} + \frac{|\epsilon|}{2} \|V'\|_{L^\infty} \|\bm\varphi\|_{L^2} \\
&\leq \frac{|\epsilon| a}{|d|} \|V\|_{L^\infty} \|L_1 \bm\varphi\|_{L^2} 
+ |\epsilon| \left( \left(\frac{a \mu}{|d|} +\frac{1}{4a} \right) 
\|V\|_{L^\infty} + \frac{1}{2} \|V'\|_{L^\infty} \right) \|\bm\varphi\|_{L^2}
\end{align*}
where we used the inequality
$$
\forall \bm\varphi \in X \times X, \,\forall a>0:\quad \|\bm\varphi'\|_{L^2} \leq a \|\bm\varphi''\|_{L^2} + \frac{1}{4a} \|\bm\varphi\|_{L^2}.
$$
Hence, by the dissipative perturbation theorem, cf.~Chapter III, Theorem~2.7 in \cite{Engel_Nagel}, for generators the operator $L_1+L_2 : X \times X \to Z \times Z$ generates a contraction semigroup. 

\medskip

Part (iii): It follows that $L_3$ is bounded on $Z \times Z$. Then the bounded perturbation theorem for generators, cf.~Chapter III, Theorem~1.3 in \cite{Engel_Nagel}, yields that $\widetilde{L}_{u(\epsilon),\epsilon} = L_1+L_2+L_3$ generates a $C_0$-semigroup on $Z \times Z$ as desired.
\end{proof}

\begin{Remark}
Using similar arguments, one can show that $\widetilde{L}_{u(\epsilon),\epsilon}$ is the generator of a $C_0$-semigroup on $Y \times Y$.
\end{Remark}

For step (ii), we use a characterization of exponential decay of semigroups in Hilbert spaces known as the Gearhart-Greiner-Pr\"uss~Theorem, cf.~Chapter V, Theorem~1.11 in \cite{Engel_Nagel}.

\begin{Theorem}[Gearhart-Greiner-Pr\"uss Theorem]\label{Thm:GP}
Let $L$ be the generator of a $C_0$-semigroup $(\eu^{Lt})_{t\geq 0}$ on a complex Hilbert space $H$. Then $(\eu^{Lt})_{t\geq 0}$ is exponentially stable in $H$ if and only if 
$$
\{\lambda \in \C:\RT (\lambda) \geq 0\} \subset \rho(L)
\quad\text{and}\quad
\sup_{\RT \lambda \geq 0} \|(\lambda I-L)^{-1}\|_{H \to H} < \infty.
$$
\end{Theorem}

By the assumption of Theorem~\ref{thm:nonlinear_stability}, spectral stability of the solution $u(\epsilon)$ is guaranteed and we are left with the proof of the uniform resolvent estimate on $\{\lambda\in \C:\RT(\lambda)\geq 0\}$. Using Lemma~\ref{lem:resolvent_est}, we find $\lambda^* \gg 1$ such that $(\lambda I-\widetilde{L}_{u(\epsilon),\epsilon})^{-1}$ is uniformly bounded on the set $\Lambda_{\lambda^*}$ for sufficiently small $\epsilon$. Moreover, since $\widetilde{L}_{u(\epsilon),\epsilon}$ is the generator of a $C_0$-semigroup on the state-space $Z \times Z$, the Hille-Yosida~Theorem ensures a uniform bound of the resolvent on $\{\lambda \in \C : \RT(\lambda)>\lambda_*\}$ for some constant $\lambda_*>0$. From the fact that $\lambda \mapsto (\lambda I-\widetilde{L}_{u(\epsilon),\epsilon})^{-1}$ is a meromorphic function with no poles in $\{ \lambda \in \mathbb{C}: \; \RT(\lambda) \geq 0 \}$,
the resolvent is uniformly bounded on compact subsets of $\C$ in 
$\{ \lambda \in \mathbb{C}: \; \RT(\lambda) \geq 0 \}$. Thus, we can conclude that $\widetilde{L}_{u(\epsilon),\epsilon}$ satisfies the Gearhart-Greiner-Pr\"uss resolvent bound and exponential stability in $Z \times Z$ follows.

\medskip

Finally, for step (iii), we will interpolate the decay estimate between the spaces $Z \times Z$ and $X \times X$. To do so, we have to establish bounds in $X \times X$ which is done the next lemma. The interpolation argument is then in the spirit of Lemma~5 in \cite{Stanislavova_Stefanov} and will also lead to decay estimates in the more general interpolation spaces $H_\per^s \times H_\per^s$ for $s \in [0,2]$. 

\begin{Lemma}\label{lem:Lin_Stability_Hs}
For any $s\in [0,2]$ and sufficiently small $\epsilon$ the semigroup $(\eu^{\widetilde{L}_{u(\epsilon),\epsilon} t})_{t\geq 0}$ has exponential decay in $H_\per^s([-\pi,\pi],\C) \times H_\per^s([-\pi,\pi],\C)$, i.e., there exist $C_s> 0$ such that
$$
\|\eu^{\widetilde{L}_{u(\epsilon),\epsilon} t}\|_{H^s \to H^s} \leq C_s \eu^{-\eta t} \quad\text{for } t \geq 0,
$$
where $-\eta<0$ is the previously established growth bound of the semigroup in $Z \times Z$.
\end{Lemma}

\begin{proof}
We consider only the case $d>0$, since the other case can be shown by rewriting $JA_u$ as $-J(-A_u)$ and using the same arguments as presented below. If $d>0$, the operator $A_{u(\epsilon)} + \gamma I$ is positive and self-adjoint provided $\gamma>0$ is sufficiently large. Hence, for $z \in \C$ we can define the complex powers by
$$
(A_{u(\epsilon)} + \gamma I)^z \bm v = \int_0^\infty  \lambda^{z} d E_\lambda \bm v,  \quad \text{for }\bm v \in \dom(A_u{u(\epsilon)} + \gamma I)^z,
$$
with domain given by
$$
\dom(A_{u(\epsilon)} + \gamma I)^z = \left\{ \bm v \in Z \times Z : \|(A_{u(\epsilon)} + \gamma I)^z\bm v\|_{L^2}^2 = \int_0^\infty  \lambda^{2 \RT z} d \|E_\lambda \bm v\|_{L^2}^2 < \infty \right\}
$$
and where $E_\lambda$ for $\lambda \in \R$ is the family of self-adjoint spectral projections associated to $A_{u(\epsilon)} + \gamma I$. Note that for $\theta \in [0,1]$ the relation 
$$
\dom(A_{u(\epsilon)} + \gamma I)^\theta = H_\per^{2\theta}([-\pi,\pi],\C) \times H_\per^{2\theta}([-\pi,\pi],\C)
$$
is true, cf.~\cite{Lunardi} Theorem~4.36, and further for any $r \in \R$ the operator $(A_{u(\epsilon)} + \gamma I)^{\iu r}$ is unitary on $Z \times Z$. If $\theta = 0,1$ we will show that there exists $C_\theta >0$ such that
\begin{align*}
\forall r \in \R,\, \forall t \geq 0, \,\forall \bm v \in X \times X,:\quad
\|(A_{u(\epsilon)}+\gamma I)^{\theta + \iu r} \eu^{\widetilde{L}_{u(\epsilon),\epsilon} t} \bm v\|_{L^2} \leq C_\theta \eu^{-\eta t} \|\bm v\|_{H^{2\theta}}, 
\end{align*}
which implies
$$
\forall r \in \R, \,\forall t \geq 0,\, \forall \theta \in (0,1),\,\forall \bm v \in X \times X:
\;\|(A_{u(\epsilon)}+\gamma I)^{\theta + \iu r} \eu^{\widetilde{L}_{u(\epsilon),\epsilon} t} \bm v\|_{L^2} \leq C_0^{1-\theta} C_1^{\theta} \eu^{-\eta t} \|\bm v\|_{H^{2\theta}},
$$
by complex interpolation, cf.~\cite{Lunardi} Theorem~2.7. In particular, we see that
$$
\|\eu^{\widetilde{L}_{u(\epsilon),\epsilon} t}\|_{H^{s} \to H^s} \leq C_0^{1-s} C_1^s \eu^{-\eta t}
$$
which is precisely our claim. The estimate for $\theta=0$ has already been shown in the preceding discussion, so it remains to check the estimate for $\theta=1$. Let $\bm v \in X \times X$ and observe that
\begin{align*}
\|(A_{u(\epsilon)}+\gamma I)^{1 + \iu r } \eu^{\widetilde{L}_{u(\epsilon),\epsilon} t} \bm v\|_{L^2} &= \|(A_{u(\epsilon)}+\gamma I)\eu^{\widetilde{L}_{u(\epsilon),\epsilon} t} \bm v\|_{L^2}\\
&=\|(\widetilde{L}_{u(\epsilon),\epsilon} + J\gamma +I(\mu - \epsilon V(x) \partial_x)) \eu^{\widetilde{L}_{u(\epsilon),\epsilon} t} \bm v \|_{L^2} \\
&\leq \| \eu^{\widetilde{L}_{u(\epsilon),\epsilon} t} \widetilde{L}_{u(\epsilon),\epsilon}\bm v\|_{L^2} + C \| \eu^{\widetilde{L}_{u(\epsilon),\epsilon}t} \bm v\|_{L^2} + 
|\epsilon| \|V\|_{L^\infty} \|\partial_x \eu^{\widetilde{L}_{u(\epsilon),\epsilon}t} \bm v\|_{L^2}  \\
&\leq C \eu^{-\eta t} \|\widetilde{L}_{u(\epsilon),\epsilon} \bm v\|_{L^2} + C \eu^{-\eta t} \|\bm v\|_{L^2} + |\epsilon| \|V\|_{L^\infty} \|\eu^{\widetilde{L}_{u(\epsilon),\epsilon}t} \bm v\|_{H^1} \\
&\leq C \eu^{-\eta t} \|\bm v\|_{H^2} + |\epsilon|  C \|\eu^{\widetilde{L}_{u(\epsilon),\epsilon}t} \bm v\|_{H^2},
\end{align*}
which yields $\|(A_{u(\epsilon)}+\gamma I)^{1 + \iu r }\eu^{\widetilde{L}_{u(\epsilon),\epsilon} t} \bm v\|_{L^2} \leq C \eu^{-\eta t} \|\bm v\|_{H^2}$ if $\epsilon$ is sufficiently small because of the norm equivalence $\|\bm v\|_{H^2} \sim \|(A_{u(\epsilon)}+\gamma I)\bm v\|_{L^2}$.
\end{proof}

In particular Lemma~\ref{lem:Lin_Stability_Hs} establishes exponential stability of the linearization in $Y \times Y$, thus we have proved Theorem~\ref{thm:nonlinear_stability}.

\subsection{Proof of Lemma~\ref{lem:resolvent_est}}\label{sec:proof_resolvent_est}

The uniform resolvent estimate is proved if we can find a constant $C>0$ independent of $\lambda \in \Lambda_{\lambda^*}$ such that
\begin{equation}\label{eq:resolvent_reversed}
\forall  \bm\varphi \in X \times X:\quad \|(\lambda I - \widetilde{L}_{u(\epsilon),\epsilon}) \bm\varphi\|_{L^2} \geq C \|\bm\varphi\|_{L^2}.
\end{equation}
In order to simplify the situation, let us introduce the rotation on $Z \times Z$ as follows:
$$
R 
\begin{pmatrix}
\varphi_1 \\ \varphi_2
\end{pmatrix}:=
\begin{pmatrix}
\cos\theta & \sin\theta \\ -\sin\theta & \cos\theta
\end{pmatrix}
\begin{pmatrix}
\varphi_1 \\ \varphi_2
\end{pmatrix}
$$
with spatially varying angular $\theta(x) = \frac{\epsilon}{2d} \int_{-\pi}^x [V(y) - \hat{V}_0]dy$ where $\hat{V}_0 = \frac{1}{2\pi} \int_{-\pi}^{\pi} V(y) \,dy$ is the mean of the potential $V$. Since $R$ is an isometry on $Z \times Z$ the resolvent estimate \eqref{eq:resolvent_reversed} is equivalent to
$$
\forall  \bm\varphi \in X \times X:\quad \|(\lambda I - R\widetilde{L}_{u(\epsilon),\epsilon} R^{-1}) \bm\varphi\|_{L^2} \geq C \|\bm\varphi\|_{L^2},
$$
where we note that $\sigma(\widetilde{L}_{u(\epsilon),\epsilon}) = \sigma(R\widetilde{L}_{u(\epsilon),\epsilon}R^{-1})$.
The advantage of considering the operator $R\widetilde{L}_{u(\epsilon),\epsilon}R^{-1}$ becomes clear if we calculate
$$
R\widetilde{L}_{u(\epsilon),\epsilon}R^{-1} = J \tilde{A}_{u(\epsilon),\epsilon,V} - I (\mu - \epsilon \hat{V}_0 \partial_x) 
$$
where the operator $\tilde{A}_{u(\epsilon),\epsilon,V}$ given by
$$
\tilde{A}_{u(\epsilon),\epsilon,V} := 
\begin{pmatrix}
-d\partial_x^2 + W_1  & W_2+W_4 \\
W_2-W_4 & -d\partial_x^2 + W_3
\end{pmatrix}
$$
with potentials
\begin{align*}
W_1 &= \zeta +\cos^2\theta U_1 + 2 \cos\theta \sin\theta U_2 + \sin^2\theta U_3 + d \theta'^2 - \epsilon\theta' V,\\
W_2 &= (\cos^2\theta - \sin^2\theta) U_2 + \cos\theta \sin\theta (U_3-U_1),\\
W_3 &= \zeta + \sin^2\theta U_1 - 2 \cos\theta \sin\theta U_2 + \cos^2\theta U_3 + d \theta'^2 - \epsilon\theta' V ,\\
W_4 &= d \theta'',
\end{align*}
and functions
\begin{align*}
U_1 =-( 3 u_1^2(\epsilon) + u_2^2(\epsilon)), \,\,
U_2 = -2 u_1(\epsilon) u_2(\epsilon), \,\,
U_3 = -(u_1^2(\epsilon)+3u_2^2(\epsilon)).
\end{align*}
Clearly, the first order derivative is now multiplied by a constant instead of a spatially varying potential which will be used in the following calculations. We also note that the functions $W_i \in X$, $i=1,2,3$ depend upon the solution $u$ and the potential $V$ whereas $W_4 \in X$ only depends upon the potential $V$. For the proof of the resolvent estimate we use techniques presented in \cite{Stanislavova_Stefanov}, where the authors construct resolvents for the unperturbed LLE \eqref{LLE_original}. 

We need the following proposition, which is Lemma~4 in \cite{Stanislavova_Stefanov}.

\begin{Proposition}\label{prop:estimate_for_determinant}
Let $d\not=0$ and $\mu > 0$. Then there exists $\lambda^*>0$ depending on $d$ and $\mu$ with the property that for all $\omega \geq \lambda^* $ there is at most one $k_0=k_0(\omega,\mu) \in \N$ such that
$$
\omega \geq  |d^2k_0^4+{\mu}^2 -\omega^2 |.
$$
For all other $k \in \Z \setminus \{\pm k_0(\omega,\mu)\}$ we have
$$
|d^2k^4  +{\mu}^2 -\omega^2 | \geq \frac{1}{10} \max\{d^2k^2,\omega\}^{3/2}.
$$
Moreover, we find $k_0(\omega,\mu) = \mathcal{O}(\omega^{1/2})$ as $\omega \to \infty$.
\end{Proposition}

%====================================================
% Skip proof since it is presented in Stanislavova
%====================================================
\begin{comment}
\begin{proof}
It is enough to consider the case $k\geq 0$ and $|d|=1$. If $\lambda > 100$ then we find that the inequality
$$
\lambda \geq |-\lambda^2 +\mu^2 + k^4| \geq |-\lambda +k^2| (\lambda + k^2)
$$
implies $ 1 \geq |k^2 - \lambda| = |k-\sqrt{\lambda}|(k + \sqrt{\lambda})$ which then gives $|k-\sqrt{\lambda}| \leq \sqrt{\lambda}^{-1} \leq 10^{-1}$. Let $k_0 = k_0(\lambda) \in \N$ denote the integer satisfying $|k-\sqrt{\lambda}| \leq 10^{-1}$, if it exists. Then for $k \not = k_0$ we find
\begin{align*}
|-\lambda^2 +\mu^2 + k^4| &\geq |-\lambda +k^2| (\lambda + k^2)\\
&= |k - \sqrt{\lambda}| (k + \sqrt{\lambda}) (\lambda + k^2) \\
&\geq \frac{1}{10} (k + \sqrt{\lambda}) (\lambda + k^2) \\
&\geq \frac{1}{10} \max\{k^3, \lambda^{3/2}\}.
\end{align*}
\end{proof}

\end{comment}
Now we can start to construct and bound the resolvent. By the Hille-Yoshida Theorem, a uniform resolvent estimate holds whenever $\RT\lambda$ is sufficiently large. It therefore remains to consider $\lambda = \delta + \iu \omega \in \Lambda_{\lambda^*}$ for some $\lambda^*>0$ and $\delta \geq 0$ on a compact set. Since $\delta$ replaces $\mu$ 
	in $\lambda I - \widetilde{L}_{u(\epsilon),\epsilon}$ by $\mu + \delta$ and the estimates of Proposition \ref{prop:estimate_for_determinant} holds for any $\mu > 0$ on a compact set, it sufficies to prove the uniform 
estimates for $\delta = 0$. For now, we do not specify the value of $\lambda^*$, since this will be done later in the proof. We can restrict to the case $\omega \geq \lambda^*$, since the proof for $\omega \leq -\lambda^*$ follows from symmetries of the spectral problem under complex conjugation. For $\bm v \in X \times X$ we define
\begin{equation}\label{eq:resolvent_est}
(\lambda I- R\widetilde{L}_{u(\epsilon),\epsilon} R^{-1}) \bm v =: \bm\psi \in Z \times Z
\end{equation}
and show that there exist bounded operators $T_1$ and $T_2$ on $Z \times Z$ depending on $\lambda$ with norms satisfying $\|T_1\|_{L^2 \to L^2} = \mathcal{O}(\omega^{-1/2})$ and $\|T_2\|_{L^2 \to L^2}=\mathcal{O}(1)$ as $\omega \to \infty$ such that \eqref{eq:resolvent_est} implies
\begin{equation}\label{eq:neumann_form}
(I+T_1) \bm v = T_2 \bm\psi.
\end{equation}
If $\lambda^*$ is sufficiently large, we then deduce that $I+T_1$ is a small perturbation of the identity, and hence invertible with norm uniformly bounded in $\lambda$ which is our claim. Therefore, it remains to show \eqref{eq:neumann_form}. We introduce the matrix-valued potential
$$
W = 
\begin{pmatrix}
W_1 & W_2+W_4 \\
W_2 - W_4 & W_3
\end{pmatrix} 
$$ 
in order to write
$$
\lambda I- R\widetilde{L}_{u(\epsilon),\epsilon} R^{-1} = \iu \omega I- J(-d \partial_x^2 + W
) + I (\mu - \epsilon \hat{V}_0 \partial_x).
$$
Now, let $A = \lambda I- R\widetilde{L}_{u(\epsilon),\epsilon} R^{-1} + JW$ and observe that $A\bm v(x) = \sum_{k \in \Z} A_k \hat{\bm v}_k \eu^{\iu k x}$ with $\bm v (x) = \sum_{k \in \Z} \hat{\bm v}_k \eu^{\iu k x}$ and Fourier multiplier
$$
A_k = A_k^1 + A_k^2 =
\begin{pmatrix}
\iu \omega + \mu & -dk^2 \\
dk^2 & \iu \omega + \mu
\end{pmatrix}
+
\begin{pmatrix}
- \iu\epsilon \hat{V}_0k & 0 \\
0 & - \iu \epsilon \hat{V}_0k
\end{pmatrix}.
$$
The inverse of $A_k^1$ is given by
$$
(A_k^1)^{-1} = \frac{1}{\text{det($A_k^1$)}} 
\begin{pmatrix}
\iu \omega + \mu & dk^2 \\
-dk^2 & \iu \omega + \mu
\end{pmatrix}
$$
and by Proposition~\ref{prop:estimate_for_determinant} there exists at most one $k_0 = k_0(\omega,\mu) \in \N$ such that
$$
|\mathrm{det}(A_k^1)| \geq |d^2k^4 + \mu^2 -\omega^2| \geq \frac{1}{10} \max\{d^2k^2,\omega\}^{3/2} \text{ for all $k \not= \pm k_0$}
$$
provided that $\lambda^*$ is sufficiently large. Thus $A_k^1$ is invertible with bound $\|(A_k^1)^{-1}\|_{\C^{2\times 2}} \leq C / \max\{\omega^{1/2}, k\}$ for all $k \not = \pm k_0$. Using again Proposition~\ref{prop:estimate_for_determinant}, we have the asymptotic $k_0 = k_0(\omega) = \mathcal{O}(\omega^{1/2})$ as $\omega \to \infty$. Consequently, if $|\epsilon|$ is sufficiently small, then $A_k=A_k^1(I+(A_k^1)^{-1} A_k^2)$, $k\not= \pm k_0$, is also invertible with the bound $\|(A_k)^{-1}\|_{\C^{2\times 2}} =\mathcal{O}(\omega^{-1/2})$ as $\omega \to \infty$. Next, for the above $k_0 \in \N$, we introduce the orthogonal projections $P, Q, Q_1, Q_2: Z\times Z\to Z\times Z$ as follows: 
$$
Q_1 \bm v = \hat{\bm v}_{k_0} \eu^{\iu k_0(\cdot)}, \quad Q_2 \bm v = \hat{\bm v}_{-k_0} \eu^{-\iu k_0(\cdot)}
$$
and 
$$
Q = Q_1+ Q_2, \quad P=I-Q.
$$
This allows us to decompose \eqref{eq:resolvent_est} as follows:
\begin{align}
PAP \bm v - P JW \bm v  &=  P \bm\psi, \label{eq:split_1} \\
Q A Q \bm v - Q JW \bm v  &= Q \bm\psi. \label{eq:split_2}
\end{align}
From the preceding arguments we find
$$
\|(PAP)^{-1}\|_{L^2 \to L^2} = \mathcal{O}(\omega^{-1/2}) \text{ as }\omega\to \infty
$$ 
which implies that \eqref{eq:split_1} is equivalent to
\begin{align}\label{eq:proof_step1}
P \bm v  - (PAP)^{-1} P JW \bm v  =  (PAP)^{-1} \bm\psi
\end{align}
with bound $\|(PAP)^{-1} JW\|_{L^2 \to L^2} = \mathcal{O}(\omega^{-1/2})$ as $\omega \to \infty$.

\medskip

Next we investigate \eqref{eq:split_2} which we decompose a second time to find
\begin{align}
Q_1 A Q_1 \bm v - Q_1  JW Q_1 \bm v - Q_1 JW Q_2 \bm v - Q_1 JW P \bm v &= Q_1 \bm\psi, \label{eq:split_2_1}\\
Q_2 A Q_2 \bm v - Q_2 JW Q_1 \bm v - Q_2 JW Q_2 \bm v - Q_2 JW P \bm v &= Q_2 \bm\psi. \label{eq:split_2_2}
\end{align}
Both equations can be handled similarly and thus we focus on the first one. Using \eqref{eq:proof_step1} we can write \eqref{eq:split_2_1} as
$$
[Q_1 A Q_1  - Q_1  JW Q_1 ] \bm v - Q_1 JW Q_2 \bm v - Q_1 JW (PAP)^{-1} P JW \bm v = Q_1 JW (PAP)^{-1} \bm\psi + Q_1 \bm\psi.
$$
The operator $B:=Q_1 A Q_1  - Q_1  JW Q_1$ acts like a Fourier-multiplier on $\range Q_1$ with matrix
$$
B_{k_0} =
\begin{pmatrix}
\iu (\omega - \epsilon \hat{V}_0 k_0) +\mu - (\hat{W_2})_{0} + (\hat{W_4})_{0}  & -dk_0^2 - (\hat{W_3})_{0} \\
dk_0^2 + (\hat{W_1})_0 & \iu (\omega - \epsilon \hat{V}_0 k_0) + \mu + (\hat{W_2})_{0} + (\hat{W_4})_{0}
\end{pmatrix}
$$
and we observe that
$$
|\mathrm{det}(B_{k_0})| \geq |\IT \mathrm{det}(B_{k_0})|  = 2 |\omega - \epsilon \hat{V}_0 k_0| |\mu + \hat{(W_4)}_0|  \sim \omega 
$$
since $k_0 = \mathcal{O}(\omega^{1/2})$ and $\omega \gg 1$. This means that $B_{k_0}$ is invertible with $\|B_{k_0}^{-1}\|_{\C^{2 \times 2}}$ uniformly bounded in $\omega \gg 1$, and thus the same holds for the operator $B$. Inverting $B$ yields
\begin{align*}
Q_1 \bm v - B^{-1}[Q_1 JW Q_2 + Q_1 JW (PAP)^{-1} P JW]\bm v = B^{-1} Q_1 JW (PAP)^{-1} \bm\psi + B^{-1}Q_1 \bm\psi
\end{align*}
and since we have $W_i \in Y$ for $i=1,2,3,4$ we can exploit decay of the Fourier-coefficients
$$
|(\hat{W_i})_{k}| \leq \frac{C}{\sqrt{1+k^2}} \quad\text{for all $k\in \Z$}
$$
to bound $Q_1  JW Q_2\bm v = (\hat{JW})_{2k_0} \hat{\bm v}_{-k_0} \eu^{\iu k_0 (\cdot)}$:
$$
\|Q_1  JW Q_2\|_{L^2 \to L^2} = \mathcal{O}(k_0(\omega,\mu)^{-1}) = \mathcal{O}(\omega^{-1/2}) \text{ as } \omega \to \infty.
$$
Finally from the bounds of the first part we infer that
\begin{align*}
\|Q_1 JW (PAP)^{-1} P JW\|_{L^2 \to L^2} &= \mathcal{O}(\omega^{-1/2}) \text{ as }\omega \to \infty,\\
\|Q_1 JW (PAP)^{-1}\|_{L^2 \to L^2} &= \mathcal{O}(\omega^{-1/2}) \text{ as }\omega \to \infty
\end{align*}
and as a conclusion we arrive at \eqref{eq:neumann_form} which is all we had to prove.

\appendix
\section{Derivation of the perturbed LLE}
\label{appA}

%The following derivation was shown to us by Dr. Huanfa Peng, Institute of Photonics and Quantum Electronics (IPQ) at Karlsruhe Institue of Technology.
The following is a derivation of the perturbed LLE \eqref{TWE_dyn} from the dual laser pump equation \eqref{LLE_dual}. We start by taking a solution $u=u(x,t)$ of \eqref{LLE_dual}. Jumping in a moving coordinate system we set $\tilde u(x,t)= u(k_1x-\nu_1 t,t)$ and find that $\tilde u$ satisfies
\begin{equation} \label{first_step}
\iu \partial_t \tilde u - \iu \nu_1 \partial_{\xi} \tilde u = -dk_1^2 \partial_{\xi}^2 \tilde u+ (-\iu\mu +\zeta)\tilde u - |\tilde u|^2\tilde u+ \iu f_0 +\iu f_1 \eu^{\iu \xi},
\end{equation}
where $\xi := k_1 x - \nu_1 t$.
Next, using the approximation $\arctan s \approx s$ for $|s|$ small, we find for $|f_0|\gg |f_1|$ that
$$
f_0+f_1\eu^{\iu \xi} = \sqrt{f_0^2+2f_0f_1\cos \xi + f_1^2} \eu^{\iu \arctan\frac{f_1 \sin \xi}{f_0+f_1\cos \xi}} \approx f_0 \eu^{\iu \frac{f_1}{f_0}\sin \xi}.
$$
Inserting this into \eqref{first_step} we find that approximately the following equation holds for $\tilde u$
\begin{equation} \label{second_step}
\iu \partial_t \tilde u - \iu \nu_1 \partial_{\xi} \tilde u = -dk_1^2 \partial_{\xi}^2 \tilde u+(-\iu\mu +\zeta)\tilde u  - |\tilde u|^2\tilde u+ \iu f_0 \eu^{\iu \frac{f_1}{f_0}\sin \xi}.
\end{equation}
This suggests to set $\tilde u(\xi,t)=w(\xi,t) \eu^{\iu \frac{f_1}{f_0}\sin \xi}$ so that $w$ solves
\begin{align*}
\iu \partial_t w =&  -dk_1^2 \partial_{\xi}^2 w + \left(\iu \nu_1-\iu 2d k_1^2 \frac{f_1}{f_0}\cos \xi\right) \partial_{\xi} w \\
& + \Bigl(-\iu\mu+\zeta \underbrace{- \nu_1\frac{f_1}{f_0}\cos \xi + dk_1^2\frac{f_1^2}{f_0^2}\cos^2 \xi +\iu dk_1^2 \frac{f_1}{f_0}\sin \xi}_{=:\alpha(\xi)}\Bigr) w -|w|^2 w+\iu f_0.
\end{align*}
Using $|f_1|\ll |f_0|$ we see that the term $\alpha(\xi)$ is much smaller than $-\iu\mu+\zeta$ for physically relevant (normalized) values of $\mu=\mathcal{O}(1)$ and $\zeta$ between $\mathcal{O}(1)$ and $\mathcal{O}(10)$. Neglecting $\alpha(\xi)$ we arrive at 
$$
\iu \partial_t w = -dk_1^2 \partial_{\xi}^2 w+\iu \bigl(\underbrace{\nu_1-2d k_1^2 \frac{f_1}{f_0}\cos \xi}_{=:V(\xi)}\bigr) \partial_{\xi} w + (-\iu\mu+\zeta) w -|w|^2 w+\iu f_0
$$
which is our target equation \eqref{TWE_dyn} in the case $\epsilon=1$ and with $d$ replaced by $d k_1^2$.

\section{Stability criterion for solitary waves in the limit of small $\mu$}
\label{appB}

The stability criterion of Theorem \ref{thm:spectral_stability} becomes more explicit in the limit $\mu \to 0$ for solitary waves on $\mathbb{R}$ for the focusing case $d > 0$. We thus consider the stationary LLE in the form:
\begin{equation}
\label{TWE-line}
-d u''+(\zeta-\mathrm{i}\mu)u-|u|^2 u +\iu \mu f_0=0, \qquad x \in \mathbb{R}.
\end{equation}
Here both the pumping term $\mathrm{i} \mu f_0$ and the dissipative term $-\mathrm{i} \mu u$ are small and of equal order in $\mu$. When $\mu$ is small, the solution can be expanded asymptotically as
\begin{equation}
\label{expansion}
u = u^{(0)} + \mu u^{(1)} + \mathcal{O}(\mu^2).
\end{equation}
Here $u^{(0)}$ is the solitary wave of the nonlinear Schr\"{o}dinger equation (NLSE) which exists if $d > 0$ and $u^{(1)}$ is found from the linear inhomogeneous equation 
\begin{equation}
\label{lin-mu-correction}
(-d \partial_x^2 + \zeta-2|u^{(0)}|^2) u^{(1)} - (u^{(0)})^2 \bar{u}^{(1)} = \mathrm{i} u^{(0)} +\iu f_0.
\end{equation}
By using the vector form with $u = u_1 + \mathrm{i} u_2$ and the linearization operator $\widetilde{L}_u = J A_u - \mu I $ as in (\ref{decomposition}), we can rewrite (\ref{lin-mu-correction}) in the form: 
$J A_{u^{(0)}} \bm{u}^{(1)} = \bm{u}^{(0)} +f_0$. Recall that 
$$
\ker \widetilde{L}_{u} = \spann\{\bm{u}'\}, \quad \ker  \widetilde{L}_{u}^* = \spann\{J \bm{\phi}^*\},
$$
according to Assumption (A2), which implies that 
$$
J A_u \bm{u}' = \mu \bm{u}', \qquad J A_u \bm{\phi}^* = -\mu \bm{\phi}^*.
$$
Expansion (\ref{expansion}) yields at the order of $\mathcal{O}(\mu)$ that 
\begin{align*}
\bm{u}' &= (\bm{u}^{(0)})' + \mu (\bm{u}^{(1)})' + \mathcal{O}(\mu^2), \\
\bm{\phi}^* &= C \left[ (\bm{u}^{(0)})' + \mu [(\bm{u}^{(1)})' + 2 \bm{v}^{(1)}] + \mathcal{O}(\mu^2) \right],
\end{align*}
where $\bm{v}^{(1)}$ is a solution of the linear inhomogeneous equation 
$J A_{u^{(0)}} \bm{v}^{(1)} = -(\bm{u}^{(0)})'$ and the constant 
$C= C(\mu) \in \mathbb{C}$ is found from the normalization condition 
$\langle \bm u', J\bm\phi^* \rangle_{L^2} = 1$. The solution of $J A_{u^{(0)}} \bm{v}^{(1)} = -(\bm{u}^{(0)})'$ on the line $\mathbb{R}$ is available explicitly:
\begin{equation*}
\bm{v}^{(1)} = - \frac{1}{2d} x J \bm{u}^{(0)},
\end{equation*}
where $\bm{u}^{(0)}(x) \to 0$ as $|x| \to \infty$ exponentially fast in the 
case of solitary waves for $d > 0$. This allows us to compute by using  integration by parts:
\begin{align*}
\langle \bm u', J\bm\phi^* \rangle_{L^2} &= C \left[ 2\mu \int_{\mathbb{R}} 
[(u_1^{(0)})' v_2^{(1)} - (u_2^{(0)})' v_1^{(1)}] dx + \mathcal{O}(\mu^2) \right] \\
&= C \left[ \mu d^{-1} \int_{\mathbb{R}} x
[(u_1^{(0)})' u_1^{(0)} + (u_2^{(0)})' u_2^{(0)}] dx + \mathcal{O}(\mu^2) \right] \\
&= C \left[ -\frac{\mu}{2d} \| u^{(0)} \|^2_{L^2} + \mathcal{O}(\mu^2) \right]. 
\end{align*}
Normalization $\langle \bm u', J\bm\phi^* \rangle_{L^2} = 1$ defines 
$C$ asymptotically as follows:
$$
C = - \frac{2d}{\mu \| u^{(0)} \|^2_{L^2}} \left[ 1 + \mathcal{O}(\mu) \right].
$$
The stability condition of Theorem \ref{thm:spectral_stability} is expressed in terms of the sign of $V'_{\mathrm{eff}}(\sigma_0)$, where $\sigma_0$ is a simple root of $V_{\mathrm{eff}}$. The effective potential can now be written more explicitly as 
\begin{align*}
V_{\mathrm{eff}}(\sigma_0) &= \langle V(\cdot + \sigma_0) \bm u', J\bm\phi^* \rangle_{L^2} \\
&= C \left[ \mu d^{-1} \int_{\mathbb{R}} x V(x + \sigma_0)
[(u_1^{(0)})' u_1^{(0)} + (u_2^{(0)})' u_2^{(0)}] dx + \mathcal{O}(\mu^2) \right] \\
&= \frac{1}{\| u^{(0)} \|^2_{L^2}} \int_{\mathbb{R}} [x V'(x+\sigma_0) + V(x + \sigma_0)] |u^{(0)}|^2 dx
 + \mathcal{O}(\mu).
\end{align*}
If $V_{\mathrm{eff}}(\sigma_0) = 0$, then 
the solitary wave of the stationary LLE (\ref{TWE-line}) with small $\mu = 0$ 
is uniquely continued in the perturbed equation for small $\epsilon$ and 
the unique continuation is spectrally stable if $V'_{\mathrm{eff}}(\sigma_0) \cdot \epsilon > 0$.

\section*{Acknowledgments} L.~Bengel and W.~Reichel acknowledge funding by the Deutsche Forschungsgemeinschaft (DFG, German Research Foundation) -- Project-ID 258734477 -- SFB 1173. D. E. Pelinovsky acknowledges support by the Alexander von Humboldt Foundation as Humboldt Reseach Award. 

We are grateful to Dr. Huanfa Peng, Institute of Photonics and Quantum Electronics (IPQ) at Karlsruhe Institue of Technology for showing us how to derive our main equation \eqref{TWE_dyn} from the two-mode pumping variant \eqref{LLE_dual} of the LLE.

\bibliographystyle{plainurl}
\bibliography{bibliography}

\begin{thebibliography}{10}

\bibitem{Alama}
Stan Alama, Lia Bronsard, Andres Contreras, and Dmitry~E. Pelinovsky.
\newblock Domains walls in the coupled gross-pitaevskii equations.
\newblock {\em Archive Rational Mech Appl.}, 215:579--615, 2015.

\bibitem{Cazenave}
Thierry Cazenave and Alain Haraux.
\newblock {\em An introduction to semilinear evolution equations}, volume~13 of
  {\em Oxford Lecture Series in Mathematics and its Applications}.
\newblock The Clarendon Press, Oxford University Press, New York, 1998.
\newblock Translated from the 1990 French original by Yvan Martel and revised
  by the authors.

\bibitem{hara_delcey}
Lucie Delcey and Mariana Haragus.
\newblock Instabilities of periodic waves for the {L}ugiato-{L}efever equation.
\newblock {\em Rev. Roumaine Math. Pures Appl.}, 63(4):377--399, 2018.

\bibitem{DelHara_periodic}
Lucie Delcey and Mariana Haragus.
\newblock Periodic waves of the {L}ugiato-{L}efever equation at the onset of
  {T}uring instability.
\newblock {\em Philos. Trans. of the Roy. Soc. A}, 376(2117):20170188, 2018.
\newblock \href {https://doi.org/10.1098/rsta.2017.0188}
  {\path{doi:10.1098/rsta.2017.0188}}.

\bibitem{DOH14}
Tomas Dohnal, Jens Rademacher, Hannes Uecker, and Daniel Wetzel.
\newblock pde2path 2.0: multi-parameter continuation and periodic domains.
\newblock 2014.

\bibitem{Engel_Nagel}
Klaus-Jochen Engel and Rainer Nagel.
\newblock {\em One-parameter semigroups for linear evolution equations}, volume
  194 of {\em Graduate Texts in Mathematics}.
\newblock Springer-Verlag, New York, 2000.
\newblock With contributions by S. Brendle, M. Campiti, T. Hahn, G. Metafune,
  G. Nickel, D. Pallara, C. Perazzoli, A. Rhandi, S. Romanelli and R.
  Schnaubelt.

\bibitem{Sigal}
J~Fr\"{o}hlich, S.~Gustafson, B.~L.~G. Jonsson, and I.~M. Sigal.
\newblock Solitary wave dynamics in an external potential.
\newblock {\em Comm. Math. Phys.}, 250:613--642, 2004.

\bibitem{Gaertner_et_al}
J.~G\"artner, P.~Trocha, R.~Mandel, C.~Koos, T.~Jahnke, and W.~Reichel.
\newblock Bandwidth and conversion efficiency analysis of dissipative kerr
  soliton frequency combs based on bifurcation theory.
\newblock {\em Phys. Rev. A}, 100:033819, Sep 2019.
\newblock URL: \url{https://link.aps.org/doi/10.1103/PhysRevA.100.033819},
  \href {https://doi.org/10.1103/PhysRevA.100.033819}
  {\path{doi:10.1103/PhysRevA.100.033819}}.

\bibitem{gaertner_reichel_waves}
Janina G\"{a}rtner and Wolfgang Reichel.
\newblock Soliton solutions for the {L}ugiato--{L}efever equation by analytical
  and numerical continuation methods.
\newblock In Willy D\"{o}rfler, Marlis Hochbruck, Dirk Hundertmark, Wolfgang
  Reichel, Andreas Rieder, Roland Schnaubelt, and Birgit Sch\"{o}rkhuber,
  editors, {\em Mathematics of Wave Phenomena}, Trends in Mathematics, pages
  179--195. Birkh\"{a}user Basel, oct 2020.
\newblock \href {https://doi.org/10.1007/978-3-030-47174-3_11}
  {\path{doi:10.1007/978-3-030-47174-3_11}}.

\bibitem{Gasmi_Jahnke_Kirn_Reichel}
Elias Gasmi, Tobias Jahnke, Michael Kirn, and Wolfgang Reichel.
\newblock Global continua of solutions to the {L}ugiato--{L}efever model for
  frequency combs obtained by two-mode pumping.
\newblock CRC 1173 Preprint 2022/56, Karlsruhe Institute of Technology, oct
  2022.
\newblock URL:
  \url{https://www.waves.kit.edu/downloads/CRC1173_Preprint_2022-56.pdf}, \href
  {https://doi.org/10.5445/IR/1000151945} {\path{doi:10.5445/IR/1000151945}}.

\bibitem{Gasmi_Peng_Koos_Reichel}
Elias Gasmi, Huanfa Peng, Christian Koos, and Wolfgang Reichel.
\newblock Bandwidth and conversion-efficiency analysis of {K}err soliton combs
  in dual-pumped resonators with anomalous dispersion.
\newblock CRC 1173 Preprint 2022/55, Karlsruhe Institute of Technology, oct
  2022.
\newblock URL:
  \url{https://www.waves.kit.edu/downloads/CRC1173_Preprint_2022-55.pdf}, \href
  {https://doi.org/10.5445/IR/1000151944} {\path{doi:10.5445/IR/1000151944}}.

\bibitem{Godey_2017}
Cyril Godey.
\newblock A bifurcation analysis for the {L}ugiato-{L}efever equation.
\newblock {\em The European Physical Journal D}, 71(5):131, May 2017.
\newblock \href {https://doi.org/10.1140/epjd/e2017-80057-2}
  {\path{doi:10.1140/epjd/e2017-80057-2}}.

\bibitem{Godey_et_al2014}
Cyril Godey, Irina~V. Balakireva, Aur\'elien Coillet, and Yanne~K. Chembo.
\newblock Stability analysis of the spatiotemporal {L}ugiato-{L}efever model
  for {K}err optical frequency combs in the anomalous and normal dispersion
  regimes.
\newblock {\em Phys. Rev. A}, 89:063814, 2014.
\newblock URL: \url{http://link.aps.org/doi/10.1103/PhysRevA.89.063814}, \href
  {https://doi.org/10.1103/PhysRevA.89.063814}
  {\path{doi:10.1103/PhysRevA.89.063814}}.

\bibitem{Hakkaev_Stefanov}
Sevdzhan Hakkaev, Milena Stanislavova, and Atanas~G. Stefanov.
\newblock On the generation of stable {K}err frequency combs in the
  {L}ugiato-{L}efever model of periodic optical waveguides.
\newblock {\em SIAM J. Appl. Math.}, 79(2):477--505, 2019.
\newblock \href {https://doi.org/10.1137/18M1192767}
  {\path{doi:10.1137/18M1192767}}.

\bibitem{hara_johns_perk}
Mariana Haragus, Mathew~A. Johnson, and Wesley~R. Perkins.
\newblock Linear modulational and subharmonic dynamics of spectrally stable
  {L}ugiato-{L}efever periodic waves.
\newblock {\em J. Differential Equations}, 280:315--354, 2021.
\newblock \href {https://doi.org/10.1016/j.jde.2021.01.028}
  {\path{doi:10.1016/j.jde.2021.01.028}}.

\bibitem{hara_johns_perk_derijk}
Mariana Haragus, Mathew~A. Johnson, Wesley~R. Perkins, and Björn de~Rijk.
\newblock Nonlinear modulational dynamics of spectrally stable
  {L}ugiato-{L}efever periodic waves, 2021.
\newblock URL: \url{https://arxiv.org/abs/2106.01910}, \href
  {https://doi.org/10.48550/ARXIV.2106.01910}
  {\path{doi:10.48550/ARXIV.2106.01910}}.

\bibitem{Kato}
Tosio Kato.
\newblock {\em Perturbation theory for linear operators}.
\newblock Classics in Mathematics. Springer-Verlag, Berlin, 1995.
\newblock Reprint of the 1980 edition.

\bibitem{Kielhoefer}
Hansj\"{o}rg Kielh\"{o}fer.
\newblock {\em Bifurcation theory}, volume 156 of {\em Applied Mathematical
  Sciences}.
\newblock Springer, New York, second edition, 2012.
\newblock An introduction with applications to partial differential equations.
\newblock \href {https://doi.org/10.1007/978-1-4614-0502-3}
  {\path{doi:10.1007/978-1-4614-0502-3}}.

\bibitem{Lugiato_Lefever1987}
L.~A. Lugiato and R.~Lefever.
\newblock Spatial dissipative structures in passive optical systems.
\newblock {\em Phys. Rev. Lett.}, 58:2209--2211, 1987.
\newblock URL: \url{http://link.aps.org/doi/10.1103/PhysRevLett.58.2209}, \href
  {https://doi.org/10.1103/PhysRevLett.58.2209}
  {\path{doi:10.1103/PhysRevLett.58.2209}}.

\bibitem{Lunardi}
Alessandra Lunardi.
\newblock {\em Interpolation theory}.
\newblock Appunti. Scuola Normale Superiore di Pisa (Nuova Serie). [Lecture
  Notes. Scuola Normale Superiore di Pisa (New Series)]. Edizioni della
  Normale, Pisa, second edition, 2009.

\bibitem{Mandel}
Rainer Mandel and Wolfgang Reichel.
\newblock A priori bounds and global bifurcation results for frequency combs
  modeled by the {L}ugiato-{L}efever equation.
\newblock {\em SIAM J. Appl. Math.}, 77(1):315--345, 2017.
\newblock \href {https://doi.org/10.1137/16M1066221}
  {\path{doi:10.1137/16M1066221}}.

\bibitem{marin2017microresonator}
Pablo Marin-Palomo, Juned~N Kemal, Maxim Karpov, Arne Kordts, Joerg Pfeifle,
  Martin~HP Pfeiffer, Philipp Trocha, Stefan Wolf, Victor Brasch, Miles~H
  Anderson, et~al.
\newblock Microresonator-based solitons for massively parallel coherent optical
  communications.
\newblock {\em Nature}, 546(7657):274--279, 2017.

\bibitem{Miyaji_Ohnishi_Tsutsumi2010}
T.~Miyaji, I.~Ohnishi, and Y.~Tsutsumi.
\newblock Bifurcation analysis to the {L}ugiato-{L}efever equation in one space
  dimension.
\newblock {\em Phys. D}, 239(23-24):2066--2083, 2010.
\newblock URL: \url{http://dx.doi.org/10.1016/j.physd.2010.07.014}, \href
  {https://doi.org/10.1016/j.physd.2010.07.014}
  {\path{doi:10.1016/j.physd.2010.07.014}}.

\bibitem{Parra-Rivas2018}
Pedro Parra-Rivas, Dami\`{a} Gomila, Lendert Gelens, and Edgar Knobloch.
\newblock Bifurcation structure of localized states in the {L}ugiato-{L}efever
  equation with anomalous dispersion.
\newblock {\em Phys. Rev. E}, 97(4):042204, 2018.
\newblock URL:
  \url{https://journals.aps.org/pre/abstract/10.1103/PhysRevE.97.042204}, \href
  {https://doi.org/10.1103/PhysRevE.97.042204}
  {\path{doi:10.1103/PhysRevE.97.042204}}.

\bibitem{Parra-Rivas2014}
Pedro Parra-Rivas, Dami\`{a} Gomila, Fran\c{c}ois Leo, St\'{e}phane Coen, and
  Lendert Gelens.
\newblock Third-order chromatic dispersion stabilizes {K}err frequency combs.
\newblock {\em Opt. Lett.}, 39(10):2971--2974, 2014.
\newblock URL: \url{http://ol.osa.org/abstract.cfm?URI=ol-39-10-2971}, \href
  {https://doi.org/10.1364/OL.39.002971} {\path{doi:10.1364/OL.39.002971}}.

\bibitem{Parra-Rivas2016}
Pedro Parra-Rivas, Edgar Knobloch, Dami\`{a} Gomila, and Lendert Gelens.
\newblock {Dark solitons in the {L}ugiato-{L}efever equation with normal
  dispersion}.
\newblock {\em Phys. Rev. A}, 93(6):1--17, 2016.
\newblock URL:
  \url{https://journals.aps.org/pra/abstract/10.1103/PhysRevA.93.063839}, \href
  {https://doi.org/10.1103/PhysRevA.93.063839}
  {\path{doi:10.1103/PhysRevA.93.063839}}.

\bibitem{Kev}
Dmitry~E. Pelinovsky and P.~G. Kevrekidis.
\newblock Dark solitons in external potentials.
\newblock {\em Z. angew. Math. Phys.}, 59:559--599, 2008.

\bibitem{Perinet}
Nicolas P{\'e}rinet, Nicolas Verschueren, and Saliya Coulibaly.
\newblock Eckhaus instability in the {L}ugiato-{L}efever model.
\newblock {\em The European Physical Journal D}, 71(9):243, Sep 2017.
\newblock \href {https://doi.org/10.1140/epjd/e2017-80078-9}
  {\path{doi:10.1140/epjd/e2017-80078-9}}.

\bibitem{picque2019frequency}
Nathalie Picqu{\'e} and Theodor~W H{\"a}nsch.
\newblock Frequency comb spectroscopy.
\newblock {\em Nature Photonics}, 13(3):146--157, 2019.

\bibitem{Stanislavova_Stefanov}
Milena Stanislavova and Atanas~G. Stefanov.
\newblock Asymptotic stability for spectrally stable {L}ugiato-{L}efever
  solitons in periodic waveguides.
\newblock {\em J. Math. Phys.}, 59(10):101502, 12, 2018.
\newblock \href {https://doi.org/10.1063/1.5048017}
  {\path{doi:10.1063/1.5048017}}.

\bibitem{Taheri_2017}
Hossein Taheri, Andrey~B. Matsko, and Lute Maleki.
\newblock Optical lattice trap for {K}err solitons.
\newblock {\em The European Physical Journal D}, 71(6), jun 2017.
\newblock URL: \url{https://doi.org/10.1140%2Fepjd%2Fe2017-80150-6}, \href
  {https://doi.org/10.1140/epjd/e2017-80150-6}
  {\path{doi:10.1140/epjd/e2017-80150-6}}.

\bibitem{trocha2018ultrafast}
Philipp Trocha, M~Karpov, D~Ganin, Martin~HP Pfeiffer, Arne Kordts, S~Wolf,
  J~Krockenberger, Pablo Marin-Palomo, Claudius Weimann, Sebastian Randel,
  et~al.
\newblock Ultrafast optical ranging using microresonator soliton frequency
  combs.
\newblock {\em Science}, 359(6378):887--891, 2018.

\bibitem{Udem2002}
Th. Udem, R.~Holzwarth, and T.~W. H{\"{a}}nsch.
\newblock {Optical frequency metrology}.
\newblock {\em Nature}, 416(6877):233--237, 2002.
\newblock URL: \url{http://www.nature.com/doifinder/10.1038/416233a}, \href
  {https://doi.org/10.1038/416233a} {\path{doi:10.1038/416233a}}.

\bibitem{UEC14}
Hannes Uecker, Daniel Wetzel, and Jens~D.M. Rademacher.
\newblock pde2path - a {M}atlab package for continuation and bifurcation in
  2{D} elliptic systems.
\newblock {\em NMTMA}, (7):58--106, 2014.

\bibitem{yang2017microresonator}
Qi-Fan Yang, Myoung-Gyun Suh, Ki~Youl Yang, Xu~Yi, and Kerry~J Vahala.
\newblock Microresonator soliton dual-comb spectroscopy.
\newblock In {\em CLEO: Science and Innovations}, pages SM4D--4. Optica
  Publishing Group, 2017.

\end{thebibliography}
	
\end{document}